\documentclass[11pt,a4paper,twoside]{article}
\usepackage{amsmath,amsfonts,amssymb,amsthm,bbm,latexsym,mathrsfs}
\usepackage{graphicx,color,epsfig,fancyhdr,dsfont}
\usepackage{enumerate}
\usepackage{hyperref}
\usepackage{indentfirst}
\usepackage[all]{hypcap}
\usepackage[affil-it]{authblk}
\allowdisplaybreaks
\newtheorem{thm}{Theorem}[section]
\newtheorem{lem}[thm]{Lemma}
\newtheorem{cor}[thm]{Corollary}
\newtheorem{prop}[thm]{Proposition}
\newtheorem{defi}[thm]{Definition} 
\newtheorem{rmk}[thm]{Remark}
\newtheorem{exmp}[thm]{Example}

\newcommand{\E}{\mathbb{E}}

\renewcommand{\theequation}{\arabic{section}.\arabic{equation}}
 \numberwithin{equation}{section}

\renewcommand{\E}{\mathbb{E}}

\allowdisplaybreaks

\begin{document}

\baselineskip15pt

\title{Anticipating Random Periodic Solutions--II. 
 SPDEs with Multiplicative Linear Noise}
 
\author[1]{
Chunrong Feng}
\author [2] {Yue Wu}
\author[1]{Huaizhong Zhao}
\affil[1]{Department of Mathematical Sciences, Loughborough
University, LE11 3TU, UK}
 \affil[2] {Institute of Mathematics, TU Berlin, Germany}
\affil[ ]{C.Feng@lboro.ac.uk, wu@math.tu-berlin.de, H.Zhao@lboro.ac.uk}
\date{}

\maketitle
\newcounter{bean}
\vskip-10pt

\begin{abstract}

In this paper, we study the existence of random periodic solutions for
semilinear stochastic partial differential equations with multiplicative linear noise on a bounded open domain ${\cal O}\subset {\mathbb R}^d$ with smooth boundary. We identify them with the solutions of coupled forward-backward infinite horizon stochastic
integral equations in $L^2({\cal O})$. We then use generalized Schauder's fixed point theorem,
 the relative compactness of Wiener-Sobolev spaces in $C^0([0, T], L^2(\Omega\times{\cal O}))$ and a localization argument
to prove the existence of solutions of the infinite horizon integral equations, which immediately implies the existence of the random periodic solution to the 
corresponding SPDEs. As an example, we apply our result to the stochastic 
Allen-Cahn equation with a periodic potential and prove the existence of a random periodic solution using a localisation argument.  

\medskip

\noindent
{\bf Keywords:} random periodic solutions, 
semilinear stochastic partial differential equations, relative compactness, Malliavin derivative, coupled forward-backward infinite horizon random integral equations, stochastic Allen-Cahn equation with a periodic potential.

\end{abstract}

 \renewcommand{\theequation}{\arabic{section}.\arabic{equation}}

\section{Introduction}\label{sec:1}

The concept of random periodic solutions of stochastic dynamical systems offers a possible perspective to investigate periodic-like long time behaviors for a wide range of stochastic (partial) differential equations. It serves as the corresponding notion of deterministic periodic solutions in the stochastic counterpart. The definition of pathwise random periodic solutions of $C^1$-cocycles was given in \cite{zh-zheng}, based on which the authors showed the existence of such periodic solutions on a cylinder using the method of random invariant set, Lyapunov exponents and the pullback of the cocycle. The authors were aware that, due to the presence of noise, a periodic curve cannot be the path of the random dynamical systems, where a `periodic curve' is referred to as a `closed curve' in deterministic dynamical systems. In other words, the idea of simply having $\omega$, a realization from the random source, fixed and drawing an analogy between deterministic dynamical systems and random dynamical systems will not provide the right concept  of a random periodic solution. One can not expect to find a path $\phi$ of the stochastic dynamical system satisfying $\phi(t,\omega)=\phi(t+\tau,\omega)$ for a positive constant $\tau$.  This can be well justified by the well-known definition of stationary solutions (c.f. \cite{ar}, \cite{du-lu-sc1}, \cite{kh-ma-si},   \cite{li-zh}, \cite{mattingly}, \cite{si2}). Note a random periodic solution degenerates to a stationary solution provided the period can be arbitrary. With this observation, Feng, Zhao and Zhou \cite{feng-zhao-zhou} proposed the following definition of random periodic solutions for stochastic semiflows which describes complex behaviour of the mixture of randomness and periodicity. Given $H$, a separable Banach space, and $\Delta:=\{(t,s)|t\geq s, t,s\in \mathbb{R}\}$, consider a stochastic semi-flow  $u$: $\Delta\times H\times\Omega\to H$, which satisfies the following standard condition
\begin{equation}
u(t,r,\omega)=u(t,s,\omega)\circ u(s,r,\omega),\ \mbox{\ for all }r\leq s\leq t, \ r,s,t\in \mathbb{R},\ \mbox{\ for\ }a.e.\ \omega\in\Omega. 
\end{equation}
\begin{defi}{\upshape (Random Periodic Solutions/Paths for Semiflows)}\label{RPSS}
A random periodic solution/path of period $\tau$ of the semi-flow $u:\Delta\times H\times\Omega\to H$ is an ${\cal F}$-measurable map $Y: \mathbb{R}\times\Omega\to H$ such that
\begin{equation}\label{eqn:RPS_def}
u(t,s, Y(s,\omega), \omega)=Y(t,\omega),\ Y(s+\tau,\omega)=Y(s, \theta_\tau \omega),\ t,s\in \mathbb{R}, \ t\geq s, \ a.e.\ \omega\in\Omega.
\end{equation}
\end{defi}

Note the definition of random periodic solution is consistent with that of deterministic periodic solution when $\omega$ disappears from the system, i.e., being a deterministic system. It is consistent with 
the definition of stationary solution when Eqn. (\ref{eqn:RPS_def}) holds regardless of the value $\tau$.

The notion of random periodic paths has been used in several
 areas such as random attractors (\cite {bates-lu-wang}), stochastic resonance (\cite{cherubini}), 
 strange attractors (\cite{lian-huang}) and climate dynamics (\cite{chekroun}). The equivalence of 
 random periodic paths and periodic measures were observed in \cite{feng-zhao2}. The ergodicity of periodic measure and the existence of pure imaginary eigenvalues of the infinitesimal 
 generator of the Markov semigroup, provide a possible  spectral analytic approach to study random periodicity 
of certain physically interesting systems.

The work on the existence of random periodic solutions to semi-linear S(P)DEs with additive noise in the literature
 includes \cite{feng-zhao-zhou} and \cite{feng-zhao}. However, the approach of identifying random periodic solutions as the solutions of forward-backward coupled infinite horizon stochastic integral equations (IHSIEs) cannot be directly applied to SDEs with multiplicative noise. Because in this case, the random periodic solution is anticipating. Thus the stochastic integral involved in IHSIEs is of the Skorokhod type. One cannot estimate the Malliavin derivative when a Skorokhod integral is involved. To overcome this difficulty, Feng, Wu and Zhao \cite{feng-wu-zhao} introduced stochastic linear evolution operator to accommodate the noise resource and identified the random periodic solutions as the solutions of the corresponding forward-backward coupled infinite horizon random integral equations (IHRIEs). The absence of stochastic integrals in the IHRIEs makes the analysis work without having to deal with the Skorokhod integral directly.  
 
In this paper, we push the idea to the infinite-dimensional setting. This is a nontrivial task.
Let ${\cal O}$ be a bounded open subset of $\mathbb{R}^d$ with smooth boundary and ${\cal L}$ be a second order differential operator with Dirichlet boundary condition on ${\cal O}$,
\begin{eqnarray}\label{1.51}
{\cal L}u={1\over 2} \sum_{i,j=1}^d {\partial \over \partial x_j} \left(a_{ij}(x){\partial u \over {\partial x_i}}\right)+c(x)u,
\end{eqnarray}
and satisfy a uniformly elliptic condition:
\vskip3pt

{\bf Condition (L)}: The coefficients $a_{ij}, c$ are smooth functions on $\bar{{\cal O}}$, $a_{ij}=a_{ji}$, and there exists a constant $\gamma>0$ such that $\sum_ {i,j=1}^d a_{ij} \xi_i\xi_j\geq \gamma |\xi|^2$ for any $\xi=(\xi_1, \xi_2,\cdots,\xi_d)\in {\mathbb{R}}^d$. 
\vskip3pt

Under the above condition, ${\cal L}$ is a self-adjoint uniformly elliptic operator so it has discrete real-valued eigenvalues $\mu_1> \mu_2> \cdots$ such that $\mu_k\to{-\infty}$ when $k\to \infty$. Denote by $\{\phi_k\in L^2({\cal O}),\ k\geq 1\}$ a complete orthonormal system of  eigenfunctions of ${\cal L}$ with corresponding eigenvalues $\mu_k,\ k\geq 1$. Here the space $L^2({\cal O})=:\mathbb H$ is a standard square integrable measurable function space vanishing on the boundary with norm $||\cdot ||_{L^2({\cal O})}$.
By a standard notation $H_0^1({\cal O})$ we denote the Sobolev space of the square integrable measurable functions having the first order weak derivative in $\mathbb H$ and vanishing at the boundary $\partial {\cal O}$. This is a Hilbert space with inner product $(u,v)=\int_{{\cal O}} u(x)v(x)dx+\int_{{\cal O}}\langle \nabla  u(x), \nabla  v(x)\rangle dx$, for any $u, v\in H_0^1({\cal O})=:\mathbb K$.

Assume that
$W_t$, $t\in \mathbb{R}$, is an $\mathbb H$-valued cylindrical Brownian motion defined on the canonical filtered Wiener space  $(\Omega,\mathcal{F},(\mathcal{F}_t)_{t\in \mathbb{R}},\mathbb{P})$, with covariance space $\mathbb K$ represented by
\begin{equation}\label{W1}
W_t:=\sum_{k=1}^{\infty}W^k_t f_k(x),\ \ t\in \mathbb{R},
\end{equation}
where $\{f_k, \ k\geq 1\}$ is a complete orthonormal basis of $\mathbb K$,
and the driving noise 
$W^k$ are mutually independent one-dimensional two-sided standard Brownian motions on $(\Omega, {\cal F}, ({\mathcal{F}}^t)_{t\in \mathbb{R}}, \mathbb{P})$, ${\cal F}_s^t:=\sigma(W^k_u-W^k_v, s\leq v\leq u\leq t,\ k\in\mathbb{N})$ and ${\cal F}^t:=\vee_{s\leq t}{\cal F}_s^t$.  Set $\theta: (-\infty,\infty)\times\Omega\to \Omega$ as a flow such that $\theta_t\omega^k(s)=W^k(t+s)-W^k(t)$, for $k\in 1,2,\cdots$. It is well known that ($\Omega, {\cal F}, \mathbb{P}, (\theta_t)_{t\in \mathbb{R}}$) is a metric dynamical system.

Consider a $\tau$-periodic semilinear SPDE of Stratonovich type with multiplicative linear noise, i.e.,
\begin{eqnarray}\label{orig}
\Bigg\{\begin{array}{l}du(t,x)=\mathcal{L}u(t,x)\,dt+F(t,u(t,x))\, dt+Bu(t,x)\circ  dW(t),  \ \ \ \ \ \ t\geq s \\
u(s)=\psi\in \mathbb H, \\
u(t)|_{\partial {\cal O}}=0.
                                     \end{array}
\end{eqnarray}
where conditions on $B$ will be specified in Section \ref{sec:2} and $F$ satisfies:\\
\\
{\bf Condition (P)}: There exists a constant $\tau>0$ such that for any $t\in \mathbb{R}$, $u\in \mathbb R,$
$$F(t,u)=F(t+\tau,u).
$$

An operator $\Phi : \mathbb{R}\times\Omega\rightarrow L(\mathbb{H})$, defined by
\begin{align}
\begin{split}\label{phi}
\Bigg\{\begin{array}{l}d\Phi(t)=\mathcal{L}\Phi(t)\,dt+B \Phi(t)\circ  d W_t,   \\
\Phi(0)=I\in L({\mathbb{H}}),
                                     \end{array}
\end{split}
\end{align}
then the mild solution of (\ref{orig}) (c.f. \cite {da-za1}) via (\ref{phi}) can be written as (c.f. \cite{mo-zh-zh}) 
\begin{align}
\begin{split}
\label{NJ21}
\bigg\{\begin{array}{l}u(t,s,\psi,\omega)(x)=\Phi(t-s, \theta_s \omega)\psi(x)+\int_s^t\Phi(t-{\hat{s}}, \theta_{\hat{s}}\omega)F(\hat{s},u({\hat{s}},s,\psi,\omega))(x)\, d{\hat{s}},  \ \ t\geq s,\\
u(s,s,\psi,\omega)(x)=\psi\in \mathbb H.
                                     \end{array}
\end{split}
\end{align}
Here $F$ is understood as $F:{\mathbb R}^1\times L^2({\cal O})\to L^2({\cal O})$ in the sense defined by the Nemytskii operator (see (\ref{zhao51})).
But if $\psi$ is anticipating, we do not know whether the mild solution of (\ref{orig}) is equivalent to (\ref{NJ21}). This is because for substitution theorem of SPDEs, more condition on Malliavin derivative is needed, see \cite{mo-zh}. To avoid this, we will discuss the random periodic solution of Eqn. (\ref{NJ21}). 

In the case of the infinite-dimensional setting, techniques in \cite{feng-wu-zhao} are not adequate. We need an explicit 
representation of the stochastic evolution operator which is given in Proposition \ref{zhao31}.  The stochastic operator is a map from $L^2({\cal O})$ to itself and still permits an 
exponential dichotomy property. The details of decomposition in $L^2({\cal O})$ are essential for this analysis, especially the boundedness by Lyapunov exponents in each direction. As a 
consequence, truncated version of the stochastic evolution operator, based on a delicate decomposition, is specified along each of the directions. 
See Eqn. (\ref{PHI-B1}) and (\ref{PHI+B1}). This is followed by truncated IHRIEs Eqn. (\ref{VN}) in infinite dimensions. 
One can expect with infinite terms involved, the complexity of analysis and computation increase substantially.

On the other hand, considering projected stochastic evolution operator along each of the directions allows suitable estimations in each reduced-dimensional subspace. This prevents from overestimating where the framework in \cite{feng-wu-zhao} fails, and helps to achieve the existence of random periodic solution to SPDE (\ref{NJ21}):
\begin{thm}[Existence of random periodic solution]\label{Main}
Assume Condition (L) and Condition (B) hold in Section \ref{sec:2}. Let $F:\mathbb{R} \times \mathbb{R}^d\to \mathbb{R}$ be a continuous map, globally
bounded with globally bounded Jacobian $\nabla F(t,\cdot)$ and satisfy Condition (P). There exists at least one random periodic solution
 of the semiflow generated by Eqn. (\ref{NJ21}).
\end{thm}
The paper is organized as follows: In Section \ref{sec:2} we specify notation and conditions for the main result. We introduce 
stochastic evolution operator, discuss multiplicative ergodic theorem in infinite dimensional setting and define IHRIEs with/without truncated evolution operator.
Two crucial results 
are proved in Section \ref{sec:4}. Theorem \ref{THM1C4M} is to inentify the random periodic solutions to SPDE (\ref{orig}) with the solutions for the IHRIE and 
Theorem \ref{Main2} is to address the existence of solutions to the IHRIE.  For this we need to estimate the Malliavin derivatives, their equi-continuity and 
Sobolev norms  thus to establish the relative compactness of a sequence in  $C([0,\tau],L^2(\Omega\times {\cal O}))$.
In Section \ref{sec:5}, we consider the stochastic Allen-Cahn equation as an example. Here $F(t,u)$ satisfies Condition (P), continuity in $t$, and a weakly dissipative condition 
\begin{eqnarray}\label{eqn4.1}
uF(t,u)\leq -M u^2+L,
\end{eqnarray}
where $M,L>0$ are constant and $M>{\sigma^2\over 2}$, $\sigma^2:=\max_i \sigma_i^2$.

 In the weakly dissipative case, 
 the pull-back convergence method (c.f. \cite{crauel-flandoli}, \cite{feng-liu-zhao}, \cite {schmalfuss}) does not work well. Although Theorem \ref{Main} cannot be applied directly to this case,
the infinite horizon stochastic integral equation method, together with the truncation and localization techniques provides a powerful tool to study random periodic solutions of SPDEs with weakly dissipative coefficients.

\section{The exponential dichotomy and IHRIEs}\label{sec:2}

Firstly, we will introduce some notations and conditions. In particular, Condition (L), Condition (B) and Condition (P), which will be adopted throughout the paper, will be given. 


From the uniformly elliptic condition, i.e., {\bf{Condition (L)}} in Section \ref{sec:1}, it is not difficult to know that $\phi_k\in \mathbb K$ and there exists a constant $C$ such that 
\begin{eqnarray}\label{eqn1.6}
||\nabla \phi_k||_{L^2({\cal O})}\leq C\sqrt {|\mu_k|}.
\end{eqnarray}
Besides, with the heat kernel $K(t,x,y)$ of the second order differential operator $\cal L$,
\begin{eqnarray}\label{T1}
(T_t\phi)(x)=\int_{{\cal O}}K(t,x,y)\phi(y)dy,
\end{eqnarray}
defines a linear operator $T_t:\mathbb H\to \mathbb H$.
And by Mercer's theorem (\cite[Theorem 3.17]{hoch}), we have
\begin{eqnarray}\label{T2}
K(t, x,y)=\sum_{k=1}^\infty e^{\mu_k t}\phi_k(x)\phi_k(y).
\end{eqnarray}
 

Now denote by $L(\mathbb{K},\mathbb{H})$ the Banach space of all linear and bounded operators $J:\mathbb{K}\to \mathbb{H}$, with the norm 
$$\|J\|=\sup_{\|v\|_{\mathbb{K}}=1}\|J(v)\|_{\mathbb{H}}.$$
Denote by $L_2(\mathbb{K},\mathbb{H})$ the Hilbert space of all Hilbert-Schmidt operators $J:\mathbb{K}\to \mathbb{H}$, given the norm
$$\|J\|_2:=\left[\sum_{k=1}^{\infty} \|J(f_k)\|^2_{\mathbb{H}}\right]^{1/2}.$$

 Let $B: \mathbb{H}\to L_2(\mathbb K, \mathbb{H})$ be a bounded linear operator such that
\begin{equation}\label{B1}
B(u)(v)=\sum_{k=1}^{\infty}\sigma_k \langle u,\phi_k\rangle \langle v,f_k\rangle \phi_k,\ \ u\in \mathbb{H}, v\in \mathbb {K}.
\end{equation}

We will need the following condition in the main results of the paper.

\noindent
{\bf Condition (B)}: 
\begin{equation}\label{SGM}
\sum_{k=1}^{\infty}\sigma_k^2<\infty.
\end{equation} 

It is not hard to show the commutative property between $T_t$ and $B$ given in the next proposition.
\begin{prop}\label{COMMUNI} The operator $T_t$ defined by (\ref{T1}) and (\ref{T2}) are commutative with $B$ defined by (\ref{B1}). Moreover, it also holds that for any $u,v\in\mathbb{H}$, 
\begin{equation}\label{COMMUNI2} 
T_tB(u)(v)(x)=B(T_t u)(v)(x).
\end{equation}
\end{prop}
\begin{proof}
For any $u\in \mathbb{H}$, $v\in \mathbb{K}$,
\begin{eqnarray*}
T_t B(u)(v)(x)&=&\int_{{\cal O}}\sum_{j=1}^{\infty}e^{\mu_j t}\phi_j(y)\phi_j(x) \sum_{k=1}^{\infty}\sigma_k \langle u,\phi_k\rangle \langle v, f_k\rangle \phi_k(y)dy\\
& =&\sum_{j=1}^{\infty}e^{\mu_j t}\sigma_j \langle u,\phi_j\rangle \langle v,f_j\rangle\phi_j(x) \\
&=&\sum_{j=1}^{\infty} \sigma_j  \langle u,e^{\mu_j t}\phi_j\rangle \langle v,f_j\rangle\phi_j(x)\\
& =&\sum_{j=1}^{\infty}\sigma_j\langle  T_t u,\phi_j\rangle \langle v,f_j\rangle \phi_j(x)\\
&=& B(T_t u)(v)(x),
\end{eqnarray*}
where $\langle \cdot, \cdot\rangle$ denotes the inner product in $\mathbb H$. 
So (\ref{COMMUNI2}) follows. 
\end{proof}

Note that Equation (\ref {orig}) generates
a semi-flow $u:\Delta\times L^2({\cal O})\times\Omega\to L^2({\cal O})$ when $F:\mathbb{R}\times\mathbb{R}\to \mathbb{R}$ is continuous in time and Lipschitz continuous in $u$ (\cite{mo-zh-zh}). As to the continuous function $F:\mathbb{R}\times \mathbb{R} \to \mathbb{R}$,  we can define the Nemytskii operator $F: \mathbb{R}\times L^2({\cal O})\to L^2({\cal O})$ with the same notation
\begin{eqnarray}\label{zhao51}
F(t, u(t))(x)&=&F(t, u(t,x)),
\end{eqnarray}
and
\begin{eqnarray}
F^i(t, u(t))(x)&=&\int_{{\cal O}} F(t, u(t))(y) \phi_i(y)dy\phi_i(x),\ x\in {\cal O},\ u\in L^2({\cal O}).
\end{eqnarray}

Next, we will give the exponential dichotomy and IHRIEs. We first define a stochastic evolution operator, Eqn. (\ref{phi}), to accommodate all the noise terms, which leads to the mild solution to (\ref{orig}) in terms of random integrals. Then the properties of this stochastic evolution operator are investigated, mainly including its representation form and exponential dichotomy. This is followed by introducing the coupled forward-backward infinite horizon random integral equation (IHRIE) based on Eqn. (\ref{phi}), which is crucial for the main result. Finally a truncated version of IHRIE is defined for convergence purpose, along with the corresponding truncated version of Eqn. (\ref{phi}). 
 
Now we introduce an operator $\Phi : \mathbb{R}\times\Omega\rightarrow L(\mathbb{H})$ defined in (\ref{phi}).
%
A useful representation of $\Phi$ below simply follows the commutative property between  $T_t$ and $B$.
\begin{prop} \label{zhao31}
Assume Condition (B) holds, then $\Phi$ defined in (\ref{phi}) has the decomposition mapping from $\mathbb H$ to $\mathbb H$ as
\begin{equation}\label{PExpo}
\Phi(t,\omega)\cdot=T_te^{BW(t)}(\cdot)=T_t\sum_{k=1}^{\infty}e^{\sigma_k W^k_t}\langle \cdot,\phi_k\rangle\phi_k.
\end{equation}  
\end{prop}
\begin{proof}
According to \cite[Theorem 1.2.3]{mo-zh-zh}, generally $\Phi$ can be written in the following representation:
 \begin{eqnarray*}
 \Phi(t,\omega)=T_t+\sum_{n=1}^{\infty}\Phi^n(t,\omega),
 \end{eqnarray*}
 where
 $$\Phi^1(t,\omega)=\int_0^tT_{t-s_1}BT_{s_1}\circ dW(s_1),\ \ \ \ \ \ \ \ \ \ \ \ \ \ \ \ \ \ \ \ \ \ \ $$
 $$\Phi^n(t,\omega)=\int_0^tT_{t-s_1}B\Phi^{n-1}(s_1,\omega)\circ dW(s_1),\ \ \ n\geq 2.$$

 With the commutative property of $B$ and $T_{\cdot}$ from Proposition \ref{COMMUNI}, the expansion above can be further derived as
 $$ \Phi^1(t,\omega)\cdot=T_tB(\cdot) W(t)=T_t\sum_{k=1}^{\infty}\sigma_k W^k_t\langle \cdot,\phi_k\rangle\phi_k,\ \ \ \ \ \ \ \ \ \ \ \ \ \ \ \ \ \ \ \ \ \ \ \ \ \ \ $$
 $$\Phi^n(t,\omega)\cdot=\frac{1}{n!}T_t(B(\cdot)W(t))^n=T_t\sum_{k=1}^{\infty}\frac{1}{n!}\sigma^n_k (W^k_t)^n\langle \cdot,\phi_k\rangle\phi_k,\ \ \ n\geq 2.$$
So for any $u\in \mathbb H$, 
\begin{eqnarray}\label{sum}
\Phi^n(t,\omega)u=T_t\sum_{n=1}^\infty \sum_{k=1}^\infty\frac{1}{n!}\sigma^n_k (W^k_t)^n\langle u,\phi_k\rangle\phi_k.
\end{eqnarray}
Define
\begin{eqnarray}\label{bm}
\tilde W_t= \sum_k {\tilde W}_t^k \phi_k(x):= \sum_k \sigma _k W_t^k \phi_k(x),
\end{eqnarray}
and by Condition (B), it is easy to see that $||\tilde W_t||_{\mathbb H}<\infty$ $\mathbb P$-a.s. Therefore
\begin{eqnarray*}
\sum_{n=1}^\infty \sum_{k=1}^\infty||\frac{1}{n!}\sigma^n_k (W^k_t)^n\langle u,\phi_k\rangle\phi_k||_{\mathbb H}&=& \sum_{n=1}^\infty \sum_{k=1}^\infty\frac{1}{n!}|(\tilde W^k_t)^n u^k|\\
&\leq & \sum_{n=1}^\infty \frac{1}{n!}\Big (\sum_{k=1}^\infty|\tilde W^k_t|^{2n}\Big)^{1\over 2}\Big (\sum_{k=1}^\infty|u^k|^{2}\Big)^{1\over 2}\\
&\leq &
\sum_{n=1}^\infty \frac{1}{n!}||\tilde W_t||_{\mathbb H}^n ||u||_{\mathbb H}
\\
&=& e^{||\tilde W_t||_{\mathbb H}}||u||_{\mathbb H}\\
&<& \infty,\ \mathbb P-a.s. 
\end{eqnarray*}
Thus we can change the order of the sum in (\ref{sum}) to get 
\begin{equation*}
\Phi(t,\omega)\cdot=T_t\cdot+T_t\sum_{k=1}^\infty \sum_{n=1}^\infty\frac{1}{n!}\sigma^n_k (W^k_t)^n\langle \cdot,\phi_k\rangle\phi_k=T_te^{BW(t)}(\cdot)=T_t\sum_{k=1}^{\infty}e^{\sigma_k W^k_t}\langle \cdot,\phi_k\rangle\phi_k,
\end{equation*}  
  where $e^{BW(t)}\cdot=\sum_{k=1}^{\infty}e^{\sigma_k W^k_t}\langle \cdot,\phi_k\rangle\phi_k$ is a well-defined operator mapping from $\mathbb H$ to $\mathbb H$.  
\end{proof}

Note that since $\Phi $ is a linear perfect cocycle, it is not hard for us to check that  multiplicative ergodic theorem (MET) for infinite dimension.
 \begin{lem}[Exponential dichotomy]\label{ExponentialDichominy}
Suppose the order of the eigenvalues of the operator $\mathcal{L}$ as $ \cdots<\mu_{m+1}<0<\mu_{m}<\dots<\mu_1$, and $\mathbb H$ has a direct sum decomposition
$$\mathbb H= \cdots\oplus E_{m+1}\oplus E_{m}\oplus\cdots\oplus E_{1},$$
where $E_k:=span\{\phi_k\}$, $k=1,2,\cdots$. Assume Condition (B) holds. Let $P^k$ be the projection of $\mathbb H$ onto $E_k$ along $\oplus_{i\neq k}E_i$. Then if $u\in \mathbb H$ such that $P^ku\neq 0$ with a $k\leq m$, we have
\begin{eqnarray*}
\mu_k=\lim_{t\to -\infty}\frac{1}{t}\log\|\Phi(t,\omega)P^ku\|_{\mathbb{H}}\ \ \ \ \mathbb{P}-a.s.,
\end{eqnarray*} 
and if $u\in \mathbb H$ such that $P^ku\neq 0$ with a $k\geq m+1$, we have
 \begin{eqnarray*}
\mu_k=\lim_{t\to +\infty}\frac{1}{t}\log\|\Phi(t,\omega)P^ku\|_{\mathbb{H}}\ \ \ \ \mathbb{P}-a.s..
\end{eqnarray*} 
Moreover,
\begin{equation}\label{comuPhi}
\Phi(t,\theta_{\hat{s}}\omega)P^k=P^k\Phi(t,\theta_{\hat{s}}\omega)
\end{equation}
and we have the following estimate:
 \begin{eqnarray}\label{412}
\Bigg\{\begin{array}{l} \|\Phi(t,\theta_{\hat{s}}\omega)P^k\|\leq C_{\Lambda}(\omega)  e^{\frac{1}{2}\mu_{k}t}e^{\Lambda|\hat{s}|}\ \mbox{for all }t\leq 0,\ \hat{s}\in \mathbb{R}, \  \mathbb{P}-a.s.,\ \ \mbox{when}\ \ k\leq m,  \\
\\
\|\Phi(t,\theta_{\hat{s}}\omega)P^k\|\leq C_{\Lambda}(\omega)   e^{\frac{1}{2}\mu_{k}t}e^{\Lambda|\hat{s}|}\ \ \ \ \mbox{for all } t\geq 0, \  \hat{s}\in \mathbb{R},\ \mathbb{P}-a.s.,\ \ \mbox{when}\ \ \ k\geq m+1,
                                     \end{array}
\end{eqnarray} 
where $\Lambda$ is an arbitrary positive number and $C_{\Lambda}(\omega)$ is a positive random variable depending on $\Lambda$.
In particular, $\mathbb H$ can be decomposed as 
$$\mathbb H=E^{-}\oplus E^{+},$$
where $E^-=\cdots\oplus E_{m+2}\oplus E_{m+1}$ is generated by the eigenvectors with negative eigenvalues, while $E^+=E_m\oplus E_{m-1}\oplus \dots\oplus E_{1}$ by the eigenvectors with positive eigenvalues.

 Let $P^{\pm}: \mathbb H\to E^{\pm}$ be the projection onto $ E^{\pm}$ along $E^{\mp}$. Then
$$\Phi(t,\theta_{\hat{s}}\omega)P^{\pm}=P^{\pm}\Phi(t,\theta_{\hat{s}}\omega),$$
with the following estimate
 \begin{eqnarray*}
\Bigg\{\begin{array}{l} \|\Phi(t,\theta_{\hat{s}}\omega)P^+\|\leq C(\theta_{\hat{s}}\omega)  e^{\frac{1}{2}\mu_{m}t}\ \ \mbox{for all }t\leq 0,\ \hat{s}\in \mathbb{R}, \  \mathbb{P}-a.s.,\\
\\
\|\Phi(t,\theta_{\hat{s}}\omega)P^-\|\leq C(\theta_{\hat{s}}\omega)  e^{\frac{1}{2}\mu_{m+1}t} \  \mbox{for all } t\geq 0, \  \hat{s}\in \mathbb{R},\ \mathbb{P}-a.s.,
                                     \end{array}
\end{eqnarray*} 
where $C(\omega)$ is a tempered random variable from above, i.e., 
$$
\lim\limits _{\hat{s}\to \pm \infty}{1\over |\hat{s}|}\log^{+} |C(\theta _{\hat{s}}\omega)|= 0,\ \  \ \mathbb{P}-a.s..
$$
 \end{lem}
  \begin{proof}
First, let us show the commutative property (\ref{comuPhi}). By using (\ref{T1}), (\ref{T2}), (\ref{COMMUNI2}) and (\ref{PExpo}) we have for any $u\in \mathbb H$,
\begin{eqnarray*}
P^k\Phi(t,\theta_{\hat{s}}\omega)(u)(x)&=&\int_{{\cal O}}K(t,y,x)P^k\big(\sum_{i=1}^{\infty}e^{\sigma_i \theta_{\hat{s}}W^i_t}\langle u,\phi_i\rangle\phi_i(y)\big)dy\\
&=&\int_{{\cal O}}e^{\mu_kt}\phi_k(y)\phi_k(x)e^{\sigma_k \theta_{\hat{s}}W^k_t}\langle u,\phi_k\rangle\phi_k(y)dy\\
&=&e^{\sigma_k \theta_{\hat{s}}W^k_t}\int_{{\cal O}}e^{\mu_kt}\phi_k(y)\phi_k(x)P^ku(y)dy\\
&=&\sum_{i=1}^{\infty}e^{\sigma_i \theta_{\hat{s}}W^i_t}\Big\langle \int_{{\cal O}}e^{\mu_kt}\phi_k(y)\phi_k(\cdot)P^ku(y)dy, \phi_i(\cdot)\Big\rangle\phi_i(x)\\
&=&\sum_{i=1}^{\infty}e^{\sigma_i \theta_{\hat{s}}W^i_t}\Big\langle (T_tP^ku)(\cdot), \phi_i(\cdot)\Big\rangle\phi_i(x)\\
&=&e^{B\theta_{\hat{s}}W(t)}(T_tP^ku)(x)\\
&=&T_te^{B\theta_{\hat{s}}W(t)}(P^ku)(x)\\
&=&\Phi(t,\theta_{\hat{s}}\omega)P^k(u)(x),
\end{eqnarray*}  
where the penultimate equality is due to the commutative property of $T_t$ and $e^{B\theta_{\hat{s}}W(t)}$. 
This implies that
$$P^k\Phi(t,\theta_{\hat{s}}\omega)=\Phi(t,\theta_{\hat{s}}\omega)P^k.$$
Moreover, when $k\geq m+1$ and $t\geq 0$, we have from the definition of $B$ and $W_t$ that,
 \begin{eqnarray}\label{LMphi}
 \Phi(t,\omega)P^k(u)(x)
 =e^{\mu_k t+\sigma_k W^k_t}\langle\phi_k(\cdot), u(\cdot)\rangle\phi_k(x)=e^{\mu_k t+\sigma_k W^k_t}u^k(x),
\end{eqnarray}  
where $u^k(x):=\langle\phi_k(\cdot), u(\cdot)\rangle\phi_k(x)$. It is not hard to deduce that for each $u\in {\mathbb{H}}$
\begin{eqnarray*}
\lim_{t\to \infty}\dfrac{1}{t}\log \|\Phi(t,\omega)P^k(u)\|_{\mathbb{H}}=\lim_{t\to \infty}\dfrac{\mu_k t+\sigma_k W^k_t}{t}+\lim_{t\to \infty}\dfrac{\log \|u\|_{\mathbb{H}}}{t}=\mu_k
\end{eqnarray*}
with probability $1$ by using strong law of large numbers for infinite-dimensional Brownian motion.
It then follows that
 \begin{eqnarray*}
 \|\Phi(t,\omega)P^k\|^2&=&\sup_{\|u\|_{\mathbb{H}}=1}\int_{{\cal O}}|\Phi(t,\omega)P^k(u)(x)|^2dx\\
 &= &\sup_{\|u\|_{\mathbb{H}}=1}e^{2\mu_k t+2\sigma_k   W_t^k} \|u^k\|_{\mathbb{H}}^2\\
  &=& e^{2\mu_k t+2\sigma_k   W_t^k} \|\phi_k\|_{\mathbb{H}}^2 \\
  &=&
  \|\Phi(t,\omega)P^k\|^2_2.
 \end{eqnarray*}
Then by Condition (B),  for arbitrary $\Lambda\in (0,|\mu_{m+1}|)$,
\begin{eqnarray*}
&&\sup_{k\geq m+1}\sup_{\hat{s}\in \mathbb{R}}\sup_{t\geq 0}\|\Phi(t,\theta_{\hat{s}}\omega)P^k\|^2e^{-\mu_k t-2\Lambda|\hat{s}|}\\
&=& \sup_{k\geq m+1}\sup_{\hat{s}\in \mathbb{R}}\sup_{t\geq 0}\exp\{\mu_k t+2\sigma_k   \theta_{\hat{s}}W_t^k-2\Lambda |\hat{s}|\}\\
&\leq& \sup_{k\geq m+1}\sup_{\hat{s}\in \mathbb{R}}\sup_{t\geq 0}\exp\{\mu_k t+2(\sum_k \sigma_k^2|\theta_{\hat{s}}W_t^k|^2)^{1\over 2}-2\Lambda |\hat{s}|\}\\
&=& \sup_{k\geq m+1}\sup_{\hat{s}\in \mathbb{R}}\sup_{t\geq 0}\exp\{\mu_k t+2||\sum_k \sigma_k\theta_{\hat{s}}W_t^k\phi_k(x)||_{\mathbb H}  -2\Lambda |\hat{s}|\}\\
&\leq & \sup_{k\geq m+1}\sup_{\hat{s}\in \mathbb{R}}\sup_{t\geq 0}\exp\{\mu_k t+2  (\| \tilde W_{t+\hat{s}}\|_{\mathbb{H}}+\| \tilde W_{\hat{s}}\|_{\mathbb{H}})-2\Lambda |\hat{s}|\}\\
&\leq & \sup_{k\geq m+1}\sup_{\hat{s}\in \mathbb{R}}\sup_{t\geq 0}\exp\{\mu_k t+\Lambda|t+\hat{s}|+2  (\| \tilde W_{t+\hat{s}}\|_{\mathbb{H}}+\| \tilde W_{\hat{s}}\|_{\mathbb{H}})-\Lambda|t+\hat{s}|-2\Lambda |\hat{s}|\}\\
&\leq &  \sup_{k\geq m+1}\sup_{\hat{s}\in \mathbb{R}}\sup_{t\geq 0}\exp\{(\mu_k+\Lambda)t+2  (\| \tilde W_{t+\hat{s}}\|_{\mathbb{H}}+\| \tilde W_{\hat{s}}\|_{\mathbb{H}})-\Lambda( |t+\hat{s}|+|\hat{s}|)\}\\
&\leq & \sup_{\hat{s}\in \mathbb{R}}\sup_{t\geq 0}\exp\{2  (\| \tilde W_{t+\hat{s}}\|_{\mathbb{H}}+\| \tilde W_{\hat{s}}\|_{\mathbb{H}})-\Lambda( |t+\hat{s}|+|\hat{s}|)\}\\
&\leq & \sup_{\hat{s}\in \mathbb{R} } e^{2 \| \tilde W_{\hat{s}}\|_{\mathbb{H}}-\Lambda|\hat{s}|}\sup_{\hat{s}+t \in \mathbb{R}} e^{2 \| \tilde W_{t+\hat{s}}\|_{\mathbb{H}}-\Lambda|t+\hat{s}|}\\
&<&\infty,\ \ \mathbb{P}-a.s.,
\end{eqnarray*}
where $\tilde W_t$ is defined in (\ref {bm}).
Similar estimation can be applied to the case when $k\leq m$ with $t\leq 0$. Then the boundedness (\ref{412}) can be drawn from those estimations. 

The definition of $E^+$ and $E^-$ naturally permits the property $\Phi(t,\theta_{\hat{s}}\omega)P^{\pm}=P^{\pm}\Phi(t,\theta_{\hat{s}}\omega)$ from (\ref{comuPhi}). Besides we know from equation (\ref{LMphi}) that
\begin{eqnarray*}
 \Phi(t,\omega)P^-(u)(x)=\sum_{k=m+1}^{\infty}\Phi(t,\omega)P^k(u)(x)=\sum_{k=m+1}^{\infty}e^{\mu_k t+\sigma_k W^k_t}u^k(x),
\end{eqnarray*}
and
$$\Phi(t,\omega)P^+(u)(x)=\sum_{k=1}^{m}e^{\mu_k t+\sigma_k W^k_t}u^k(x).$$
Thus we are able to define a new random variable
$$C(\omega):=\sqrt{C^2_1(\omega)+C^2_2(\omega)},$$
where 
\begin{eqnarray}
C_1^2(\omega):=\sup_{t\in[0,\infty)}\dfrac{\|\Phi(t,\omega)P^-\|^2}{e^{\mu_{m+1}t}}=\sup_{\|u\|_{\mathbb{H}}=1}\sup_{t\in[0,\infty)}\sum_{k=m+1}^{\infty}e^{(2\mu_k-\mu_{m+1})t+2\sigma_k   W_t^k} \|u^k\|_{\mathbb{H}}^2 \label{NJ1} \\
C_2^2(\omega):=\sup_{t\in(-\infty,0]}\frac{ \|\Phi(t,\omega)P^+\|^2}{e^{\mu_m t}}=\sup_{\|u\|_{\mathbb{H}}=1}\sup_{t\in(-\infty,0]}\sum_{k=1}^{m}e^{(2\mu_k-\mu_m) t+2\sigma_k   W_t^k} \|u^k\|_{\mathbb{H}}^2.\label{NJ1copy}
\end{eqnarray}
It remains to show that $C(\omega)$ is a tempered random variable.\\
Indeed the fact that $C^2_2(\omega)$ is tempered from above can be easly deduced by the argument in \cite{feng-wu-zhao}. Regarding the temperedness of $C_1(\omega)$, the proof mainly follows from the argument of \cite[Theorem 3.4]{car}. To this end, we apply the Kingman's subbadditive ergodic Theorem  (\cite[Theorem 3.3.2]{ar}) to $\log\|\Phi(t,\omega)P^{-}\|^2$. Thus  for every $\varepsilon>0$, there is a finite-valued random variable $C_{\varepsilon}$ such that when $n_1>n_2$,
 \begin{equation}\label{KSET2}
 \log\|\Phi(n_1-n_2,\theta_{n_2}\omega)P^{-}\|^2\leq 2(n_1-n_2)\mu_{m+1}+n\varepsilon+C_{\varepsilon}(\omega),\ \ \mathbb{P}-a.s.
 \end{equation}
Now it is sufficient to verify that 
$$\lim_{\hat{s}\to \infty}C^2_1(\theta_{\hat{s}}\omega)e^{-\bar{\Lambda} \hat{s}}=0,$$
for some sufficiently large $\bar{\Lambda}$. But this follows from 
\begin{eqnarray*}
\|\Phi(t,\theta_{\hat{\hat{s}}}\omega)P^{-}\|^2 &\leq &\|\Phi(1+t-[t]-1-[\hat{s}]+\hat{s},\theta_{[\hat{s}]+[t]}\omega)P^{-}\|^2\\
&&\times\|\Phi([t]-1,\theta_{[\hat{s}]+1}\omega)P^{-}\|^2\times \|\Phi(1-\hat{s}+[\hat{s}],\theta_{\hat{s}}\omega)P^{-}\|^2\\
&\leq &e^{D(\theta_{[\hat{s}]+[t]+1}\omega)}e^{D(\theta_{[\hat{s}]+[t]}\omega)}e^{2([t]-1)\mu_{m+1}+([t]+[\hat{s}])\varepsilon+C_{\varepsilon}(\omega)}e^{D(\theta_{[\hat{s}]}\omega)}
\end{eqnarray*} 
for $t>1$, where
$D(\omega):=\log^{+}\sup_{\hat{s},t\in [0,1]}\|\Phi(t,\theta_{\hat{s}}\omega)P^{-}\|^2$, 
and by Jensen's inequality and the Burkholder-Davis-Gundy inequality,
\begin{eqnarray*}
\mathbb{E}D&\leq &\log^+ \mathbb{E}\sup_{\hat{s},t\in [0,1]}\|\Phi(t,\theta_{\hat{s}}\omega)P^{-}\|^2\\
&=& \log^+ \mathbb{E}\sup_{\hat{s},t\in [0,1]} \sup_{\|u\|_\mathbb{H}=1}\sum_{j=m+1}^{\infty}e^{2\mu_jt+2\sigma_j \theta_{\hat{s}}W_t^j}\|u^j\|_{\mathbb{H}}^2\\
&\leq &  \log^+  \sup_{\|u\|_\mathbb{H}=1} \sum_{j=m+1}^{\infty}e^{3\sigma^2_j } \sqrt{\mathbb{E}\Big(\sup_{t\in [0,2]} e^{2\sigma_j W_{t}^j-\sigma^2_j  t}\Big)^2\mathbb{E}\Big(\sup_{\hat{s}\in [0,1]} e^{-2\sigma_j W_{\hat{s}}^j-\sigma^2_j  \hat{s}}\Big)^2}\|u^j\|_{\mathbb{H}}^2\\
&\leq &  \log^+ C_2 \sup_{\|u\|_{\mathbb{H}}=1}\sum_{j=m+1}^{\infty} e^{9\sigma^2_j } \|u^j\|_{\mathbb{H}}^2\\
&<&\infty,
\end{eqnarray*} 
here $C_2$ is a positive constant from the Burkholder-Davis-Gundy inequality. 
 \end{proof}
  \begin{rmk}
  From the proof of Lemma \ref{ExponentialDichominy}, it is easy to see that  
 \begin{eqnarray}\label{PPk1}
\Phi(t,\omega)P^k(u)(x)
&=& \langle  \phi_k(\cdot), \Phi(t,\omega)u(\cdot)\rangle \phi_k(x),
 \end{eqnarray}
and
 \begin{eqnarray}\label{PPk2}
 P^k(u)(x)=\langle u(\cdot),\phi_k(\cdot) \rangle \phi_k(x)=u^k(x).
 \end{eqnarray}
 \end{rmk}
Some elementary but useful estimates can be obtained from (\ref{412}) and  the proof is deferred to Appendix.
\begin{cor}\label{LEMMA2C4}Under the condition of Lemma \ref{ExponentialDichominy}, we have the following estimation that for an arbitrary positive $\Lambda$
\begin{equation}\label{LEMMA2C41}
\|T_tP^k-P^k\|^2\leq  |\mu_k||t|,
\end{equation}
\begin{equation}\label{LEMMA2C42}
\mathbb{E}\|\Phi_{t}P^k-T_tP^k\|^2
\leq  C\max\big\{1,e^{2\mu_kt+2\sigma^2_k |t|}\big\}\sigma_k^2 (|t|+|t|^2),
\end{equation}
for all $t\geq 0$, if $k\geq m+1$; and for all $t<0$, if $k<m$. $C$ is a generic constant.
\end{cor}

For any $N\in \mathbb{N}$, we set the truncation of $\Phi(t,\theta_{\hat{s}}\omega)P^{k}$ by $N$ according to the boundedness of $\Phi$ (\ref{412}) as follows: when $k\geq m+1$, define
\begin{equation}\label{PHI-B1}
 \Phi^N(t,\theta_{\hat{s}}\omega)P^{k}=\Phi(t,\theta_{\hat{s}}\omega)P^{k}\min\left\{1,  \frac{Ne^{\frac{1}{2}\mu_{k}t}e^{\Lambda|\hat{s}|}}{\| \Phi(t,\theta_{\hat{s}}\omega)P^{k}\|}\right\}, \mbox{\ for\ } t\geq 0,
  \end{equation}
and when $k\leq m$, define
  \begin{equation}\label{PHI+B1}
 \Phi^N(t,\theta_{\hat{s}}\omega)P^{k}=\Phi(t,\theta_{\hat{s}}\omega)P^{k}\min\left\{1,  \frac{Ne^{\frac{1}{2}\mu_{k}t}e^{\Lambda|\hat{s}|}}{\| \Phi(t,\theta_{\hat{s}}\omega)P^{k}\|}\right\}, \mbox{\ for\ } t\leq 0.
  \end{equation}
Thus we have the boundedness, i.e.,
\begin{equation*}
\| \Phi^N(t,\theta_{\hat{s}}\omega)P^{k}\|\leq N e^{\frac{1}{2}\mu_{k}t}e^{\Lambda|\hat{s}|}
\end{equation*}
for any pair of $k,t$ where $k\leq m$ and $t\geq 0$ or $k\geq m$ and $t\leq 0$.

The aim is to look for a $\mathcal{B}(\mathbb{R})\otimes \mathcal{F}\otimes \mathcal{B}({\cal O})$-measurable map $Y: \mathbb{R}\times \Omega\rightarrow \mathbb H$ which satisfies the following coupled forward-backward infinite horizon random integral equation (IHRIE) 
\begin{align}
\label{V}
\begin{split}
&Y(t, \omega)\\ 
=&\int_{-\infty}^t\Phi(t-{\hat{s}},\theta_{\hat{s}} \omega)P^-F({\hat{s}},Y({\hat{s}},\omega))d{\hat{s}}-\int^{+\infty}_t\Phi(t-{\hat{s}},\theta_{\hat{s}} \omega)P^+F({\hat{s}},Y({\hat{s}},\omega))d{\hat{s}} \\
=&\int_{-\infty}^t\sum_{k=m+1}^{\infty}\Phi(t-{\hat{s}},\theta_{\hat{s}} \omega)P^kF({\hat{s}},Y({\hat{s}},\omega))d{\hat{s}} -\int^{+\infty}_t\sum_{k=1}^{m}\Phi(t-{\hat{s}},\theta_{\hat{s}} \omega)P^kF({\hat{s}},Y({\hat{s}},\omega))d{\hat{s}}, \\
\end{split}
\end{align}
for all $\omega \in \Omega$, $t\in \mathbb{R}$. The function $Y(t, \omega)$ is a measurable function in the function space $\mathbb H$ of variable $x$, which is $Y(t, \omega)(x)$ at $x\in {\cal O}$, or written as $Y(t,\omega,x)$. Consider a sequence $\{Y^N\}_{N\geq 1}$ where $Y^N$ satisfies
\begin{align}\label{VN}
\begin{split}
 Y^N(t, \omega)=&\int_{-\infty}^t\sum_{k=m+1}^{\infty}\Phi^N(t-{\hat{s}},\theta_{\hat{s}} \omega)P^kF({\hat{s}},Y^N({\hat{s}},\omega))d{\hat{s}} \\
& -\int^{+\infty}_t\sum_{k=1}^{m}\Phi^N(t-{\hat{s}},\theta_{\hat{s}} \omega)P^kF({\hat{s}},Y^N({\hat{s}},\omega))d{\hat{s}}, 
\end{split}
\end{align}
for all $\omega \in \Omega$, $t\in \mathbb{R}$. The idea is to solve $Y^N$ on a sequence of subsets $\Omega^N$ of $\Omega$ such that $\mathbb{P}(\Omega^N)\to 1$ as $N\to \infty$. This enables us to approximate $Y$ with $Y^N$. All the arguments are deffered to Section \ref{sec:4}.

\section{The existence of random periodic solutions and periodic measures}\label{sec:4}
This section is devoted to show the main result, Theorem \ref{Main}. Section \ref{sec:4.1} gives an equivalent result with which random periodic solution can be identified as a solution to the corresponding IHRIE. Section \ref{sec:4.2} thus focuses on the existence of IHRIE, by firstly addressing the existence of the truncated version of IHRIE under generalized Schauder's fixed point argument \cite[Theorem 2.3]{feng-zhao-zhou}, and then showing convergence to the original IHRIE through localization argument. 

\subsection{Equivalence and other relevant results} \label{sec:4.1}
\begin{thm}\label{THM1C4M} Assume Condition (L), Condition (B) and Condition (P). If Cauchy problem (\ref{NJ21}) has a unique solution $u(t,s,\omega,x)$ and the coupled forward-backward infinite horizon stochastic integral equation (\ref{V}) has one solution $Y:\mathbb{R}\times \Omega\to L^2(\mathcal{O})$ such that $Y(t+\tau,\omega)=Y(t,\theta_{\tau} \omega) \ {\rm for \ any} \
t\in \mathbb{R}$ a.s., then $Y$ is a random periodic solution to (\ref{NJ21}), i.e.,
\begin{equation}\label{eq:solution}
u\big(t+\tau, t, Y(t,\omega), \omega\big)=Y(t+\tau,\omega)=Y(t,\theta_{\tau}\omega)\ \ \mbox{ for any }t\in \mathbb{R}\ \ a.s.
\end{equation}
Conversely, if Eqn. (\ref{NJ21}) has a random periodic solution $Y:\mathbb{R}\times \Omega\to L^2(\mathcal{O})$ of period $\tau$ which is tempered from above for each $t$, then $Y$ is a solution of the coupled forward-backward infinite horizon random integral equation (\ref{V}).   
\end{thm}
\begin{proof}
Similarly to \cite[Theorem 2.1]{feng-zhao}.  
\end{proof}

We will adopt the following generalized Schauder's fixed point theorem to support Theorem \ref{THMC4M}. The proof was refined from the proof of Schauder's fixed point theorem and was given in \cite[Theorem 2.3]{feng-zhao-zhou}. 
\begin{thm}[Generalized Schauder's fixed point theorem]\label{Schauder}
Let $H$ be a Banach space, S be a convex subset of $H$. Assume a map $T: H\to H$ is continuous and $T(S)\subset S$
is relatively compact in $H$. Then $T$ has a fixed point in $H$.
\end{thm}

The generalized Schauder's fixed point theorem requires us to check the relative compactness. As discussed in \cite{feng-zhao}, a suitable choice for this concern might be applying Wiener-Sobolev compact embedding theorem to get the relatively compactness of a sequence in $C([0, T], L^2(\Omega\times {\cal O}))$ with Sobolev norm being bounded in $L^2(\Omega)$ and Malliavin derivative being bounded and equicontinuous in $L^2(\Omega\times {\cal O})$ uniformly in time. 

We denote by $C_p^\infty(\mathbb{R}^n)$ the set of infinitely differentiable functions $f:\mathbb{R}^n\to \mathbb{R}$ such that $f$ and all its partial derivatives have polynomial growth. Let ${\mathcal S}$ be the class of smooth random variables $F$ such that $F=f(W(h_1),\cdots, W(h_n))$ with $n\in N$,  $h_1,\cdots, h_n\in L^2([0,T])$ and $f\in C_p^\infty(\mathbb{R}^n)$, $W(h_i)=\int_0^T h_i(s)dW(s)$. The derivative operator of a smooth random variable $F$ is the stochastic process $\{{\cal D}_t F,\  t\in [0,T]\}$ defined by (c.f. \cite{nuallart})
$${\cal D}_t F=\sum_{i=1}^n {{\partial f}\over {\partial x_i}}(W(h_1),\cdots,W(h_n))h_i(t).$$
We will denote ${\cal D}^{1,2}$ the domain of ${\cal D}$ in $L^2(\Omega)$, i.e. ${\cal D}^{1,2}$ is the closure of ${\mathcal S}$ with respect to the norm
 $$||F||_{1,2}^2=\mathbb{E}|F|^2+\mathbb{E}||{\cal D}_tF||^2_{L^2([0,T])}.$$
 Denote $C([0,T],L^2(\Omega\times {\cal O}))$ the set of continuous functions $f(\cdot,\cdot, \omega)$ with the norm
 $$||f||^2=\sup_{t\in [0,T]} \int_{{\cal O}}\mathbb{E}|f(t,x)|^2dx<\infty.$$
 The relative compactness result below is from \cite[Theorem 2.3]{feng-zhao}, a refined version of relative compactness of Wiener-Sobolev space in Bally-Saussereau \cite{bally}. 
 \begin{thm}{\bf(Relative compactness in $C([0,T], L^2(\Omega\times {\cal O}))$)}\label{B-S2}
Let ${\cal O}$ be a bounded domain in $\mathbb{R}^d$. Consider a sequence $(v_n)_{n\in N}$ of $C([0,T],L^2(\Omega\times {\cal O}))$. Suppose the following conditions are satisfied:
\begin{enumerate}
\item $\sup_{n\in \mathbb{N}}\sup_{t\in [0,T]} \mathbb{E}||v_n(t, \cdot)||_{H^1({\cal O})}^2 <\infty$,
, where the Hilbert space $H^1(\mathcal{O})$ denotes the completion of $\{v\in C^1({\cal O}), \|v\|_{H^1(\mathcal{O})}<\infty\}$ with respect to the following norm:
 $$\|v\|_{H^1(\mathcal{O})}:=\left(\|v\|^2_{L^2({\cal O})}+\|Dv\|^2_{L^2({\cal O})}\right)^{\frac{1}{2}},$$
 where $Dv$ denotes the first-order derivative of $v$ with respect to the spatial variable.  
\item $\sup_{n\in \mathbb{N}}\sup_{t\in [0,T]} \int_{{\cal O}} ||v_n(t, x,\cdot)||_{1,2}^2 dx<\infty$.
\item  There exists a constant $C\geq 0$ such that for any $t, s\in [0,T]$, 
         $$\sup_n\int_{{\cal O}} \mathbb{E}|v_n(t,x)-v_n(s,x)|^2dx< C|t-s|.$$
         \item (4i) There exists a constant $C$ such that for any $h \in \mathbb{R}$, and any $t\in [0,T]$,   
  $$ \sup_n\int_{{\cal O}}\int_{\mathbb{R}}\mathbb{E}|{\cal D}_{\theta+h}v_n(t,x)-{\cal D}_\theta v_n(t,x)|^2 d\theta dx<C|h|,$$
  where ${\cal D}_\theta$ denotes the derivative in Malliavin sense.\\
 (4ii) For any $\epsilon>0$, there exist $-\infty<\alpha<\beta<+\infty$ such that
$$\sup_n \sup_{t\in[0,T]}\int_{{\cal O}}\int_{\mathbb{R}\backslash (\alpha,\beta)} \mathbb{E}|{\cal D}_\theta v_n(t,x)|^2 d\theta dx <\epsilon.$$
\end{enumerate}
Then $\{v_n,\ n\in \mathbb{N}\}$ is relatively compact in $C([0,T],L^2(\Omega\times {\cal O}))$.
 \end{thm} 

The $\mathcal{D}^{1,2}$-norm-preserving property under the measure presearving operator is also needed to complete Schader's fixed point argument in our case. 
\begin{lem}\label{Preserving} Suppose that $F$ is an $\mathcal{F}$-measurable random variable in $L^2(\Omega\times{\cal O})$ with
\begin{eqnarray*}
F\ =\ \sum_{m=0}^{\infty}\int_{\mathbb{R}^m}f_m(t_1,\ldots,t_m,x)dW_{t_1}\ldots dW_{t_m},
\end{eqnarray*}
where $f_m(t_1,\ldots,t_m,x)$ is a symmetric function of $t_1,t_2,\cdots,t_m$ in $L^2(\mathbb{R}^m)$ for each fixed $x$. Moreover let
 $F(\cdot, x)\in \mathcal{D}^{1,2}$, then for all $h\in \mathbb{R}$, $F(\theta_h\cdot, x)\in  \mathcal{D}^{1,2}$, and
$$\int_{{\cal O}}\|F(\theta_{h}\cdot,x)\|^2_{1,2}dx=\int_{{\cal O}}\|F(\cdot,x)\|^2_{1,2}dx,$$
where the norm in $\mathcal{D}^{1,2}$ is defined as $\|f(\cdot)\|_{1,2}:=\big(\mathbb{E}|f(\cdot)|^2+\mathbb{E}\|\mathcal{D}_rf(\cdot)\|_{L^2(\mathbb{R})}^2\big)^{\frac{1}{2}}$ for $f\in \mathcal{D}^{1,2}$.
\end{lem}
\begin{proof}
The result can be proved following a similar argument as in \cite[Lemma 2.3]{feng-wu-zhao}.  
\end{proof}

\subsection{Existence of IHRIEs} \label{sec:4.2}
 The local existence theorem of random periodic solutions is presented in the following theorem.
\begin{thm}\label{THMC4M}
 Let $F:\mathbb{R} \times \mathbb{R}^d\to \mathbb{R}$ be a continuous map, globally
bounded and the Jacobian $\nabla F(t,\cdot)$ be globally bounded.
Assume $F(t,u)=F(t+\tau,u)$ for some fixed $\tau>0$, and  Condition (L) and Condition (B) hold. 
Then there exists at least one ${\cal B}(\mathbb{R})\otimes\mathcal{F}$-measurable map $Y^N: \mathbb{R}\times\Omega\rightarrow L^2({\cal O})$ 
 satisfying equation (\ref{VN}) and $Y^N(t+\tau, \omega)=Y^N(t, \theta_{\tau}\omega)$ for
any $t\in \mathbb{R}$ and $\omega\in \Omega$.
\end{thm}

\textnormal{Define a Banach space $C^{\Lambda}_{\tau}(\mathbb{R}, L^2(\Omega\times {\cal O}))$}
 \begin{eqnarray*}
 C^{\Lambda}_{\tau}(\mathbb{R}, L^2(\Omega\times {\cal O}))&:=&\{f\in C^{\Lambda}(\mathbb{R}, L^2(\Omega\times {\cal O}))\ {\rm with\ the\ following\ norm:}\\
 && \|f\|^{\Lambda}_{{\cal O}}:=\sup_{t\in \mathbb{R}}e^{-2\Lambda|t|}\|f\|_{L^2(\Omega\times {\cal O})},\\
  &&{\rm and\ for \ any} \ 
t\in \mathbb{R},
 \
f(t+\tau,\omega,x)=f(t,\theta_{\tau}\omega,x) \},
\end{eqnarray*}
where  $\Lambda< \frac{1}{4}\mu:=\frac{1}{4}\min \{-\mu_{m+1},\mu_m\}$ is a fixed real number.
Then define for any $Y^N\in  C^{\Lambda}_{\tau}(\mathbb{R}, L^2(\Omega\times {\cal O}))$  
\begin{eqnarray}\label{eqn:M}
\begin{split}
{\cal M}^N(Y^N)(t,\omega,x)&=&\int_{-\infty}^t\sum_{k=m+1}^{\infty}\Phi^N(t-{\hat{s}},\theta_{\hat{s}} \omega)P^kF({\hat{s}},Y^N({\hat{s}},\omega))(x)d{\hat{s}} \\
& &-\int^{+\infty}_t\sum_{k=1}^{m}\Phi^N(t-{\hat{s}},\theta_{\hat{s}} \omega)P^kF({\hat{s}},Y^N({\hat{s}},\omega))(x)d{\hat{s}}.
\end{split}
\end{eqnarray}
The aim of the proof is to find a fixed point of ${\cal M}^N$ in $C^{\Lambda}_{\tau}(\mathbb{R}, L^2(\Omega\times {\cal O}))$ using Schauder's fixed point theorem. 
We split it into several parts to make it more readable.
\begin{lem}\label{LEMMA3C4}Under the conditions of Theorem \ref{THMC4M},
$${\cal M}^N:C^{\Lambda}_{\tau}(\mathbb{R}, L^2(\Omega\times {\cal O})) \to C^{\Lambda}_{\tau}(\mathbb{R}, L^2(\Omega\times {\cal O})),$$
is a continuous map. Moreover, ${\cal M}^N$ maps $C^{\Lambda}_{\tau}(\mathbb{R}, L^2(\Omega\times {\cal O}))$ into $C^{\Lambda}_{\tau}(\mathbb{R}, L^2(\Omega\times {\cal O}))\cap L_{\Lambda}^{\infty}(\mathbb{R}, L^2(\Omega, H^1_0({\cal O}))),$
where $L^{\infty}_{\Lambda}\big(\mathbb{R}, L^2(\Omega,H^1_0({\cal O}))\big)$ is a Banach space of functions $f$ equipped with the following norm
 \begin{eqnarray*}
 \|f\|_{\Lambda}^{\infty}:=\inf\big\{C\geq 0: e^{-2\Lambda|t|}\|f\|_{L^2(\Omega,H^1_0({\cal O}))}\leq C\ \ \rm{for\ \ almost\ \ every}\ t\in \mathbb{R}\big\}.   
 \end{eqnarray*}
\end{lem}
\begin{proof}
{\it \textbf{Step 1}}: It is sufficient to show that ${\cal M}^N$ maps $ C^{\Lambda}_{\tau}(\mathbb{R}, L^2(\Omega\times {\cal O}))$ to itself.
\begin{enumerate}
  \item[(A)] Firstly we show that for any $Y^N\in  C^{\Lambda}_{\tau}(\mathbb{R}, L^2(\Omega\times {\cal O}))$, $$\sup\limits_{t\in \mathbb{R}}e^{-2\Lambda|t|}\mathbb{E}\int_{{\cal O}}|{\cal M}^N(Y^N)(t,\cdot,x)|^2dx<\infty.$$
  In fact, for any $t\in \mathbb{R}$, 
  \begin{eqnarray*}
  &&e^{-2\Lambda|t|}\mathbb{E}\int_{{\cal O}}|{\cal M}^N(Y^N)(t,\cdot,x)|^2dx\\
  &\leq &2e^{-2\Lambda|t|}\mathbb{E}\int_{{\cal O}}\Big|\int_{-\infty}^t\sum_{k=m+1}^{\infty}\Phi^N(t-{\hat{s}},\theta_{\hat{s}} \cdot)P^kF({\hat{s}},Y^N)(x)d{\hat{s}}\Big|^2dx\\
  &&+2e^{-2\Lambda|t|}\mathbb{E}\int_{{\cal O}}\Big|\int^{+\infty}_t\sum_{k=1}^{m}\Phi^N(t-{\hat{s}},\theta_{\hat{s}} \cdot)P^kF({\hat{s}},Y^N)(x)d{\hat{s}}\Big|^2dx\\
  & =:&J_1+J_2.
   \end{eqnarray*}
   For simplicity we shall write $\Phi^N_{t-s,s}$ instead of $\Phi^N(t-s,\theta_s \omega)$. From (\ref{PPk1}) we can estimate $J_1$ by using (\ref{412}) and H\"older inequality as follows, 
     \begin{eqnarray*} 
  J_1    &= &2e^{-2\Lambda|t|}\mathbb{E}\sum_{k=m+1}^{\infty}\int_{{\cal O}}\Big|\int_{-\infty}^t\big\langle\phi_k(\cdot), \Phi^N_{t-{\hat{s}},{\hat{s}}}P^kF^k({\hat{s}},Y^N)(\cdot)\big\rangle \phi_k(x)d{\hat{s}}\Big|^2dx\\
   &\leq &2e^{-2\Lambda|t|}\mathbb{E}\sum_{k=m+1}^{\infty}\int_{{\cal O}}\Big|\int_{-\infty}^t \|\Phi^N_{t-{\hat{s}},{\hat{s}}}P^k\||F^k({\hat{s}},Y^N)| d{\hat{s}}\Big|^2|\phi_k(x)|^2dx\\
  & \leq&2e^{-2\Lambda|t|}\mathbb{E}\sum_{k=m+1}^{\infty}\Big(\int_{-\infty}^t\|\Phi^N_{t-{\hat{s}},{\hat{s}}}P^k\||F^k({\hat{s}},Y^N)| d{\hat{s}}\Big)^2\\
  & \leq&2e^{-2\Lambda|t|}N^2\mathbb{E}\sum_{k=m+1}^{\infty}\Big(\int_{-\infty}^te^{\frac{1}{2}\mu_{k}(t-{\hat{s}})}e^{\Lambda|\hat{s}|}|F^k({\hat{s}},Y^N)| d{\hat{s}}\Big)^2,
  \end{eqnarray*}
where $|F^k(t,y)|:=\langle F(t,y)(\cdot),\phi_k(\cdot)\rangle$.
Note that $e^{\Lambda |\hat{s}|}\leq e^{\Lambda \hat{s}}+e^{-\Lambda \hat{s}}$. Based on the fact, the following trick will be use throughout the whole proof:
\begin{align}\label{eqn:trick}
\begin{split}
& e^{-\Lambda |t|}\int_{-\infty}^te^{\frac{1}{2}\mu_{k}(t-{\hat{s}})}e^{\Lambda|\hat{s}|}|F^k({\hat{s}},Y^N)| d{\hat{s}}\\
\leq & e^{-\Lambda t}\int_{-\infty}^te^{\frac{1}{2}\mu_{k}(t-{\hat{s}})}e^{\Lambda \hat{s}}|F^k({\hat{s}},Y^N)| d{\hat{s}}+e^{\Lambda t}\int_{-\infty}^te^{\frac{1}{2}\mu_{k}(t-{\hat{s}})}e^{-\Lambda \hat{s}}|F^k({\hat{s}},Y^N)| d{\hat{s}}\\
= & \int_{-\infty}^t e^{(\frac{1}{2}\mu_k-\Lambda)(t-\hat{s})}|F^k({\hat{s}},Y^N)| d{\hat{s}}+\int_{-\infty}^t e^{(\frac{1}{2}\mu_k+\Lambda)(t-\hat{s})}|F^k({\hat{s}},Y^N)| d{\hat{s}}\\
\leq & \Big(\int_{-\infty}^t e^{(\frac{1}{2}\mu_k-\Lambda)(t-\hat{s})}d{\hat{s}}\Big)^{\frac{1}{2}}\Big(\int_{-\infty}^t e^{(\frac{1}{2}\mu_k-\Lambda)(t-\hat{s})}|F^k({\hat{s}},Y^N)|^2 d{\hat{s}}\Big)^{\frac{1}{2}}\\
&+\Big(\int_{-\infty}^t e^{(\frac{1}{2}\mu_k+\Lambda)(t-\hat{s})}d{\hat{s}}\Big)^{\frac{1}{2}}\Big(\int_{-\infty}^t e^{(\frac{1}{2}\mu_k+\Lambda)(t-\hat{s})}|F^k({\hat{s}},Y^N)|^2 d{\hat{s}}\Big)^{\frac{1}{2}},
\end{split}
\end{align}
where the last line follows Cauchy-Schwarz inequality. Now substituting (\ref{eqn:trick}) to $J_1$ with $\mu_k\leq \mu_{m+1}$ yields
  \begin{eqnarray*} 
  && 2e^{-2\Lambda|t|}N^2\mathbb{E}\sum_{k=m+1}^{\infty}\Big(\int_{-\infty}^te^{\frac{1}{2}\mu_{k}(t-{\hat{s}})}e^{\Lambda|\hat{s}|}|F^k({\hat{s}},Y^N)|d{\hat{s}}\Big)^2\\
  & \leq & \frac{8N^2}{|\mu_{m+1}-2\Lambda|}\int_{-\infty}^te^{(\frac{1}{2}\mu_{m+1}-\Lambda)(t-{\hat{s}})}\mathbb{E}\sum_{k=m+1}^{\infty}|F^k({\hat{s}},Y^N)|^2 d{\hat{s}}\\
  &&+\frac{8N^2}{|\mu_{m+1}+2\Lambda|}\int_{-\infty}^te^{(\frac{1}{2}\mu_{m+1}+\Lambda)(t-{\hat{s}})}\mathbb{E}\sum_{k=m+1}^{\infty}|F^k({\hat{s}},Y^N)|^2 d{\hat{s}}\\
  &\leq & 16N^2\|F\|^2_{\infty}\Big(\frac{1}{|\mu_{m+1}-2\Lambda|^2}+\frac{1}{|\mu_{m+1}+2\Lambda|^2}\Big),
  \end{eqnarray*}
where 
\begin{equation}\label{eqn:boundedness_F}
\|F\|^2_{\infty}:=\sup_{t\in \mathbb{R},y\in \mathbb{R}}|F(t,y)|^2\geq \sup_{t\in \mathbb{R}}\sum_{k=1}^{\infty}|F^k(t,Y^N(t))|^2.
\end{equation}

Similarly we have
$$J_2\leq 16N^2\|F\|^2_{\infty}\Big(\frac{1}{|\mu_{m}-2\Lambda|^2}+\frac{1}{|\mu_{m}+2\Lambda|^2}\Big).$$
  \item[(B)]  Now it is sufficient to show that ${\cal M}^N (Y^N)(t, \omega,x)$ is continuous with respect to $t$. For this consider $t_1\leq t_2$, the following inequality simply from definition of $\mathcal{M}$, i.e., Eqn. (\ref{eqn:M}),
\begin{eqnarray*}
&&\mathbb{E}\int_{{\cal O}}|{\cal M}(Y^N)(t_1,\cdot,x)-{\cal M}(Y^N)(t_2,\cdot,x)|^2dx\\
&\leq &4\sum_{k=m+1}^{\infty}\mathbb{E}\int_{{\cal O}}\Big|\int_{-\infty}^{t_1}(\Phi^N_{t_1-{\hat{s}},{\hat{s}}}P^k-\Phi^N_{t_2-{\hat{s}},{\hat{s}}}P^k)F({\hat{s}},Y^N)(x)d{\hat{s}}\Big|^2 dx\\
&&+4\sum_{k=m+1}^{\infty}\mathbb{E}\int_{{\cal O}}\Big|\int_{t_1}^{t_2}\Phi^N_{t_2-{\hat{s}},{\hat{s}}}P^kF({\hat{s}},Y^N)(x)d{\hat{s}}\Big|^2 dx\\
&&+4\sum_{k=1}^{m}\mathbb{E}\int_{{\cal O}}\Big|\int^{+\infty}_{t_2}(\Phi^N_{t_1-{\hat{s}},{\hat{s}}}P^k-\Phi^N_{t_2-{\hat{s}},{\hat{s}}}P^k)F({\hat{s}},Y^N)(x)d{\hat{s}}\Big|^2 dx\\
&&+4\sum_{k=1}^{m}\mathbb{E}\int_{{\cal O}}\Big|\int_{t_1}^{t_2}\Phi^N_{t_1-{\hat{s}},{\hat{s}}}P^kF({\hat{s}},Y^N)(x)d{\hat{s}}\Big|^2dx\\
&=:&\sum_{i=1}^4 T_i.
\end{eqnarray*}
It is easy to check that
\begin{eqnarray*}
T_2&\leq &4\mathbb{E}\sum_{k=m+1}^{\infty}\int_{{\cal O}}\Big(\int_{t_1}^{t_2}\|\Phi^N_{t_2-{\hat{s}},{\hat{s}}}P^k\||F^k(\hat{s},Y^N)(x)|d{\hat{s}}\Big)^2dx\\
&\leq &4N^2\max\{e^{2\Lambda|t_2|},e^{2\Lambda|t_1|}\}\sum_{k=m+1}^{\infty}\int_{t_1}^{t_2}e^{\mu_k(t_2-\hat{s})}d{\hat{s}}\mathbb{E}\int_{t_1}^{t_2}\int_{{\cal O}}|F^k(\hat{s},Y^N)(x)|^2dxd\hat{s}\\
&\leq &4N^2\max\{e^{2\Lambda|t_2|},e^{2\Lambda|t_1|}\}\int_{t_1}^{t_2}e^{\mu_{m+1}(t_2-\hat{s})}d{\hat{s}}\mathbb{E}\int_{t_1}^{t_2}\sum_{k=m+1}^{\infty}\int_{{\cal O}}|F^k(\hat{s},Y^N)(x)|^2dxd\hat{s}\\
&\leq &
C_N\|F\|_{\infty}^2\max\{e^{2\Lambda|t_2|},e^{2\Lambda|t_1|}\}|t_2-t_1|^2,
\end{eqnarray*}
where $C_N$ is a general constant depending on $N$.
Similarly
\begin{eqnarray*}
T_4&:=&4\sum_{k=1}^{m}\mathbb{E}\int_{{\cal O}}\Big|\int_{t_1}^{t_2}\Phi^N_{t_1-{\hat{s}},{\hat{s}}}P^kF({\hat{s}},Y^N)(x)d{\hat{s}}\Big|^2dx\\
&\leq & C_N\|F\|_{\infty}^2\max\{e^{2\Lambda|t_2|},e^{2\Lambda|t_1|}\}|t_2-t_1|^2.
\end{eqnarray*}
For term $T_1$ we have the following estimation,
\begin{eqnarray*}
&&T_1\\
&\leq & 8\sum_{k=m+1}^{\infty}\mathbb{E}\int_{{\cal O}}\Big|\int_{-\infty}^{t_1}\min\Big\{1,\frac{Ne^{\frac{1}{2}\mu_{k}(t_1-\hat{s})}e^{\Lambda|\hat{s}|}}{\|\Phi_{t_1-{\hat{s}},{\hat{s}}}P^k\|}\Big\}
 (\Phi_{t_1-{\hat{s}},{\hat{s}}}P^k-\Phi_{t_2-{\hat{s}},{\hat{s}}}P^k)F^k({\hat{s}},Y^N)(x)d{\hat{s}}\Big|^2dx\\
&&+8\sum_{k=m+1}^{\infty}\mathbb{E}\int_{{\cal O}}\Big|\int_{-\infty}^{t_1}\Bigg( \min\Big\{1,\frac{Ne^{\frac{1}{2}\mu_{k}(t_1-\hat{s})}e^{\Lambda|\hat{s}|}}{\|\Phi_{t_1-{\hat{s}},{\hat{s}}}P^k\|}\Big\}-\min\Big\{1,\frac{Ne^{\frac{1}{2}\mu_{k}(t_2-\hat{s})}e^{\Lambda|\hat{s}|}}{\|\Phi_{t_2-{\hat{s}},{\hat{s}}}P^k\|}\Big\}\Bigg)\\
&&\hspace{3.6cm}\cdot \Phi_{t_2-{\hat{s}},{\hat{s}}}P^k F({\hat{s}},Y^N(\hat{s},\cdot))(x)d{\hat{s}}\Big|^2dx\\
&:=&T_{1,1}+T_{1,2}.
\end{eqnarray*}
Regarding the first term $T_{1,1}$, it is not hard to deduce the follows via the similar trick as in (\ref{eqn:trick}):
\begin{eqnarray*}
&&T_{1,1}\\
&\leq & 8N^2\sum_{k=m+1}^{\infty}\mathbb{E}\|\Phi_{t_2-t_1,t_1}P^k-P^k\|^2\int_{{\cal O}}\Big(\int_{-\infty}^{t_1}e^{\frac{1}{2}\mu_{k}(t_1-\hat{s})}e^{\Lambda|\hat{s}|}|F^k({\hat{s}},Y^N(\hat{s},\cdot))(x)|d{\hat{s}}\Big)^2dx\\
&\leq &\sum_{k=m+1}^{\infty} \dfrac{16 N^2 e^{2\Lambda|t_1|}}{|\mu_k-2\Lambda|}
\mathbb{E}\bigg(\|\Phi_{t_2-t_1,t_1}P^k-P^k\|^2\int_{{\cal O}}\int_{-\infty}^{t_1}e^{(\frac{1}{2}\mu_{k}-\Lambda)(t-\hat{s})}|F^k({\hat{s}},Y^N(\hat{s},\cdot))(x)|^2d\hat{s}dx\bigg)\\
&&+\sum_{k=m+1}^{\infty} \dfrac{16 N^2 e^{2\Lambda|t_1|}}{|\mu_k+2\Lambda|}
\mathbb{E}\bigg(\|\Phi_{t_2-t_1,t_1}P^k-P^k\|^2\int_{{\cal O}}\int_{-\infty}^{t_1}e^{(\frac{1}{2}\mu_{k}+\Lambda)(t-\hat{s})}|F^k({\hat{s}},Y^N(\hat{s},\cdot))(x)|^2d\hat{s}dx\bigg).
\end{eqnarray*}
Further applying the following fact  
\begin{equation}\label{eqn:PhiT}
\|\Phi_{t_2-t_1,t_1}P^k-P^k\|\leq \|\Phi_{t_2-t_1,t_1}P^k-T_{t_2-t_1}P^k\|+\|T_{t_2-t_1}P^k-P^k\|
\end{equation}
and boundedness of $F$ (\ref{eqn:boundedness_F}) leads to
\begin{align}\label{eqn:T_1,1}
\begin{split}
T_{1,1}\leq &\sum_{k=m+1}^{\infty}\dfrac{32 N^2 e^{2\Lambda|t_1|}}{|\mu_k-2\Lambda|}\|T_{t_2-t_1}P^k-P^k\|^2
\\
&
\hskip2cm 
\mathbb{E}\int_{{\cal O}}\int_{-\infty}^{t_1}e^{(\frac{1}{2}\mu_{m+1}-\Lambda)(t-\hat{s})}|F^k({\hat{s}},Y^N(\hat{s},\cdot))(x)|^2d\hat{s}dx\\
&\quad+\|F\|_{\infty}\sum_{k=m+1}^{\infty}\dfrac{32 N^2 e^{2\Lambda|t_1|}}{|\mu_k-2\Lambda|}\mathbb{E}\|T_{t_2-t_1}P^k-\Phi_{t_2-t_1,t_1}P^k\|^2\int_{-\infty}^{t_1}e^{(\frac{1}{2}\mu_{k}-\Lambda)(t-\hat{s})}
d\hat{s}\\
&\quad +\sum_{k=m+1}^{\infty}\dfrac{32 N^2 e^{2\Lambda|t_1|}}{|\mu_k+2\Lambda|}\|T_{t_2-t_1}P^k-P^k\|^2\\
&
\hskip2cm 
\mathbb{E}\int_{{\cal O}}\int_{-\infty}^{t_1}e^{(\frac{1}{2}\mu_{m+1}+\Lambda)(t-\hat{s})} |F^k({\hat{s}},Y^N(\hat{s},\cdot))(x)|^2d\hat{s}dx\\
&\quad +\|F\|_{\infty}\sum_{k=m+1}^{\infty}\dfrac{32 N^2 e^{2\Lambda|t_1|}}{|\mu_k+2\Lambda|}\mathbb{E}\|T_{t_2-t_1}P^k-\Phi_{t_2-t_1,t_1}P^k\|^2\int_{-\infty}^{t_1}e^{(\frac{1}{2}\mu_{k}+\Lambda)(t-\hat{s})}
d\hat{s}.
\end{split}
\end{align}
Now applying Corollary \ref{LEMMA2C4} to the four terms of (\ref{eqn:T_1,1}) yields
\begin{eqnarray*}
&&T_{1,1}\\
&\leq& C N^2 e^{2\Lambda|t_1|}\bigg\{\sup_{k}\dfrac{|\mu_k|}{|\mu_k-2\Lambda|}\int_{-\infty}^{t_1}e^{(\frac{1}{2}\mu_{m+1}-\Lambda)(t-\hat{s})}
\\
&&
\hskip2cm 
\mathbb{E}\int_{{\cal O}}\sum_{k=m+1}^{\infty}|F^k({\hat{s}},Y^N(\hat{s},\cdot))(x)|^2dxd\hat{s}\\
&&+\sup_{k}\dfrac{|\mu_k|}{|\mu_k+2\Lambda|}\int_{-\infty}^{t_1}e^{(\frac{1}{2}\mu_{m+1}+\Lambda)(t-\hat{s})}\mathbb{E}\int_{{\cal O}}\sum_{k=m+1}^{\infty}|F^k({\hat{s}},Y^N(\hat{s},\cdot))(x)|^2dxd\hat{s}\bigg\}|t_2-t_1|\\
&&+CN^2\|F\|_{\infty}e^{2\Lambda|t_1|}\max\Big\{\sup_{k\geq m+1}e^{2\sigma_k^2  |t_2-t_1|},1\Big\}\\
&&\hspace{0.3cm}\cdot\sum_{k=m+1}^{\infty}\Big(\dfrac{\sigma_k^2  }{|\mu_k-2\Lambda|^2}+\dfrac{\sigma_k^2 }{|\mu_k+2\Lambda|^2}\Big)(|t_2-t_1|+|t_2-t_1|^2)\\
&\leq & CN^2\|F\|_{\infty}^2e^{2\Lambda|t_1|}\Big(\dfrac{1}{|\mu_{m+1}-2\Lambda|}+\dfrac{1}{|\mu_{m+1}+2\Lambda|}\Big)\sup_{k}\dfrac{|\mu_k|}{|\mu_k-2\Lambda|}|t_2-t_1|\\
&&+CN^2\|F\|_{\infty}^2e^{2\Lambda|t_1|}\max\Big\{\sup_{k}e^{2\sigma_k^2  |t_2-t_1|},1\Big\}\Big(\sum_{k=1}^{\infty}\sigma_k^2 
\Big)\\
&&\hspace{0.3cm}\cdot\Big(\frac{1}{|\mu_{m+1}+2\Lambda|^2}+\frac{1}{|\mu_{m+1}-2\Lambda|^2}\Big)(|t_2-t_1|+|t_2-t_1|^2).
\end{eqnarray*}

With regards to term $T_{1,2}$, note that we have the following by using the inequality $|\min\{1,a\}-\min\{1,b\}|\leq |a-b|$ whenever $a,b \geq 0$:
\begin{eqnarray*}
&&\left|\min\Big\{1,\frac{N{\rm e}^{\frac{1}{2}\mu_{k}(t_1-\hat{s})}{\rm e}^{\Lambda|\hat{s}|}}{\|\Phi_{t_1-{\hat{s}},{\hat{s}}}P^k\|}\Big\}-\min\Big\{1,\frac{N{\rm e}^{\frac{1}{2}\mu_{k}(t_2-\hat{s})}{\rm e}^{\Lambda|\hat{s}|}}{\|\Phi_{t_2-{\hat{s}},{\hat{s}}}P^k\|}\Big\}\right|\\
&\leq & \left|\frac{N{\rm e}^{\frac{1}{2}\mu_{k}(t_1-\hat{s})}{\rm e}^{\Lambda|\hat{s}|}}{\|\Phi_{t_1-{\hat{s}},{\hat{s}}}P^k\|}-\frac{N{\rm e}^{\frac{1}{2}\mu_{k}(t_2-\hat{s})}{\rm e}^{\Lambda|\hat{s}|}}{\|\Phi_{t_2-{\hat{s}},{\hat{s}}}P^k\|}\right|\\
&\leq &\Bigg|\frac{N{\rm e}^{\frac{1}{2}\mu_{k}(t_1-\hat{s})}{\rm e}^{\Lambda|\hat{s}|}}{\|\Phi_{t_1-{\hat{s}},{\hat{s}}}P^k\|}-\frac{N{\rm e}^{\frac{1}{2}\mu_{k}(t_2-\hat{s})}{\rm e}^{\Lambda|\hat{s}|}}{\|\Phi_{t_1-{\hat{s}},{\hat{s}}}P^k\|}\Bigg|
+\Bigg|\frac{N{\rm e}^{\frac{1}{2}\mu_{k}(t_2-\hat{s})}{\rm e}^{\Lambda|\hat{s}|}}{\|\Phi_{t_1-{\hat{s}},{\hat{s}}}P^k\|}-\frac{N{\rm e}^{\frac{1}{2}\mu_{k}(t_2-\hat{s})}{\rm e}^{\Lambda|\hat{s}|}}{\|\Phi_{t_2-{\hat{s}},{\hat{s}}}P^k\|}\Bigg|\\
&\leq &\frac{N{\rm e}^{\Lambda|\hat{s}|}}{\|\Phi_{t_1-{\hat{s}},{\hat{s}}}P^k\|}\Big({\rm e}^{\frac{1}{2}\mu_{k}(t_1-\hat{s})}-{\rm e}^{\frac{1}{2}\mu_{k}(t_2-\hat{s})}\Big)+N{\rm e}^{\Lambda|\hat{s}|}{\rm e}^{\frac{1}{2}\mu_{k}(t_2-\hat{s})}\Bigg|\frac{\|\Phi_{t_1-{\hat{s}},{\hat{s}}}P^k\|-\|\Phi_{t_2-{\hat{s}},{\hat{s}}}P^k\|}{\|\Phi_{t_2-{\hat{s}},{\hat{s}}}P^k\|\|\Phi_{t_1-{\hat{s}},{\hat{s}}}P^k\|}\Bigg|\\
&\leq & \frac{N{\rm e}^{\Lambda|\hat{s}|}{\rm e}^{\frac{1}{2}\mu_{k}(t_1-\hat{s})}\|\Phi_{t_2-{t_1},{t_1}}P^k\|}{\|\Phi_{t_2-{t_1},{t_1}}P^k\|\|\Phi_{t_1-{\hat{s}},{\hat{s}}}P^k\|}\big(1-{\rm e}^{\frac{1}{2}\mu_{k}(t_2-t_1)}\big)\\
&&+\frac{N{\rm e}^{\Lambda|\hat{s}|}{\rm e}^{\frac{1}{2}\mu_{k}(t_2-\hat{s})}}{\|\Phi_{t_2-{\hat{s}},{\hat{s}}}P^k\|}\frac{\|\Phi_{t_1-{\hat{s}},{\hat{s}}}P^k-\Phi_{t_2-{\hat{s}},{\hat{s}}}P^k\|}{\|\Phi_{t_1-{\hat{s}},{\hat{s}}}P^k\|}\\
&\leq & \frac{N{\rm e}^{\Lambda|\hat{s}|}{\rm e}^{\frac{1}{2}\mu_{k}(t_1-\hat{s})}}{\|\Phi_{t_2-{\hat{s}},{\hat{s}}}P^k\|}\Big(\|\Phi_{t_2-{t_1},{t_1}}P^k-T_{t_2-{t_1},{t_1}}P^k\|+\big(1-{\rm e}^{\frac{1}{2}\mu_{k}(t_2-t_1)}\big)\|T_{t_2-{t_1},{t_1}}P^k\|\\
&&\quad +\|T_{t_2-{t_1},{t_1}}P^k-P^k\|\Big)\\
&:=& \frac{N{\rm e}^{\Lambda|\hat{s}|}{\rm e}^{\frac{1}{2}\mu_{k}(t_1-\hat{s})}}{\|\Phi_{t_2-{\hat{s}},{\hat{s}}}P^k\|}\big(\sum_{i=1}^3 I_i \big).
\end{eqnarray*}
Thus by substituting the inequality above into $T_{1,2}$ we have
\begin{eqnarray*}
T_{1,2}&\leq & 24\sum_{i=1}^3\sum_{k=m+1}^{\infty}\mathbb{E}\Big[|I_i|^2\int_{{\cal O}}\Big|\int_{-\infty}^{t_1}\frac{N{\rm e}^{\Lambda|\hat{s}|}{\rm e}^{\frac{1}{2}\mu_{k}(t_1-\hat{s})}}{\|\Phi_{t_2-{\hat{s}},{\hat{s}}}P^k\|} \Phi_{t_2-{\hat{s}},{\hat{s}}}P^k F({\hat{s}},Y^N(\hat{s},\cdot))(x)d{\hat{s}}\Big|^2dx\Big],
\end{eqnarray*}
where we can estimate the boundedness one by one. The term involving $I_1$ can be written as
\begin{eqnarray*}
&&24\sum_{k=m+1}^{\infty}\mathbb{E}\Big[|I_1|^2\int_{{\cal O}}\Big|\int_{-\infty}^{t_1}\frac{N{\rm e}^{\Lambda|\hat{s}|}{\rm e}^{\frac{1}{2}\mu_{k}(t_1-\hat{s})}}{\|\Phi_{t_2-{\hat{s}},{\hat{s}}}P^k\|} \Phi_{t_2-{\hat{s}},{\hat{s}}}P^k F({\hat{s}},Y^N(\hat{s},\cdot))(x)d{\hat{s}}\Big|^2dx\Big]\\
&\leq & CN^2\|F\|^2_{\infty}\sum_{k=m+1}^{\infty}\mathbb{E}
\|\Phi_{t_2-{t_1},{t_1}}P^k-T_{t_2-{t_1},{t_1}}P^k\|^2\Big(\int_{-\infty}^{t_1} {\rm e}^{\Lambda|\hat{s}|}{\rm e}^{\frac{1}{2}\mu_{k}(t_1-\hat{s})}d\hat{s}\Big)^2\\
&= & CN^2\|F\|^2_{\infty} e^{2\Lambda|t_1|}\sum_{k=m+1}^{\infty}\mathbb{E}
\|\Phi_{t_2-{t_1},{t_1}}P^k-T_{t_2-{t_1},{t_1}}P^k\|^2 e^{-2\Lambda|t_1|}\Big(\int_{-\infty}^{t_1} {\rm e}^{\Lambda|\hat{s}|}{\rm e}^{\frac{1}{2}\mu_{k}(t_1-\hat{s})}d\hat{s}\Big)^2\\
&\leq &CN^2\|F\|_{\infty}^2e^{2\Lambda|t_1|}\max\Big\{\sup_{k}e^{2\sigma_k^2  |t_2-t_1|},1\Big\}\Big(\sum_{k=1}^{\infty}\sigma_k^2 
\Big)\\
&&\hspace{0.3cm}\cdot\Big(\frac{1}{|\mu_{m+1}+2\Lambda|^2}+\frac{1}{|\mu_{m+1}-2\Lambda|^2}\Big)(|t_2-t_1|+|t_2-t_1|^2)
\end{eqnarray*}
where the last line comes from (\ref{LEMMA2C42}) and the trick (\ref{eqn:trick}). The term involving $I_2$ can be dealt with using trick (\ref{eqn:trick}):
\begin{eqnarray*}
&&24\sum_{k=m+1}^{\infty}\mathbb{E}\Big[|I_2|^2\int_{{\cal O}}\Big|\int_{-\infty}^{t_1}\frac{N{\rm e}^{\Lambda|\hat{s}|}{\rm e}^{\frac{1}{2}\mu_{k}(t_1-\hat{s})}}{\|\Phi_{t_2-{\hat{s}},{\hat{s}}}P^k\|} \Phi_{t_2-{\hat{s}},{\hat{s}}}P^k F({\hat{s}},Y^N(\hat{s},\cdot))(x)d{\hat{s}}\Big|^2dx\Big]\\
&\leq & \sum_{k=m+1}^{\infty}\dfrac{C N^2 e^{2\Lambda|t_1|}}{|\mu_k-2\Lambda|}\|T_{t_2-t_1}P^k\|^2\big(1-{\rm e}^{\frac{1}{2}\mu_{k}(t_2-t_1)}\big)^2\\
&&\quad \quad \quad \cdot \mathbb{E}\int_{{\cal O}}\int_{-\infty}^{t_1}e^{(\frac{1}{2}\mu_{k}-\Lambda)(t-\hat{s})}|F^k({\hat{s}},Y^N(\hat{s},\cdot))(x)|^2d\hat{s}dx\\
&&+\sum_{k=m+1}^{\infty}\dfrac{C N^2 e^{2\Lambda|t_1|}}{|\mu_k+2\Lambda|}\|T_{t_2-t_1}P^k\|^2\big(1-{\rm e}^{\frac{1}{2}\mu_{k}(t_2-t_1)}\big)^2\\
&&\quad \quad \quad \cdot\mathbb{E}\int_{{\cal O}}\int_{-\infty}^{t_1}e^{(\frac{1}{2}\mu_{k}+\Lambda)(t-\hat{s})}|F^k({\hat{s}},Y^N(\hat{s},\cdot))(x)|^2d\hat{s}dx\\
&\leq & C N^2 e^{2\Lambda|t_1|}\sum_{k=m+1}\bigg\{\dfrac{|\mu_{k}|^2}{|\mu_{k}-2\Lambda|}\int_{-\infty}^{t_1}e^{(\frac{1}{2}\mu_{k}-\Lambda)(t-\hat{s})}\mathbb{E}\int_{{\cal O}}^{\infty}|F^k({\hat{s}},Y^N(\hat{s},\cdot))(x)|^2dxd\hat{s}\\
&&+\dfrac{|\mu_{k}|^2}{|\mu_{k}+2\Lambda|}\int_{-\infty}^{t_1}e^{(\frac{1}{2}\mu_{k}+\Lambda)(t-\hat{s})}\mathbb{E}\int_{{\cal O}}\sum_{k=m+1}^{\infty}|F^k({\hat{s}},Y^N(\hat{s},\cdot))(x)|^2dxd\hat{s}\bigg\}|t_2-t_1|^2\\
&\leq & C N^2 e^{2\Lambda|t_1|}\|F\|^2_{\infty}\sup_{k}\Big(\dfrac{|\mu_{k}|^2}{|\mu_{k}-2\Lambda|^2}+\dfrac{|\mu_{k}|^2}{|\mu_{k}+2\Lambda|^2}\Big)|t_2-t_1|^2.
\end{eqnarray*}
The boundedness for the term involving $I_3$ is as follows
\begin{eqnarray*}
&&24\sum_{k=m+1}^{\infty}\mathbb{E}\Big[|I_3|^2\int_{{\cal O}}\Big|\int_{-\infty}^{t_1}\frac{N{\rm e}^{\Lambda|\hat{s}|}{\rm e}^{\frac{1}{2}\mu_{k}(t_1-\hat{s})}}{\|\Phi_{t_2-{\hat{s}},{\hat{s}}}P^k\|} \Phi_{t_2-{\hat{s}},{\hat{s}}}P^k F({\hat{s}},Y^N(\hat{s},\cdot))(x)d{\hat{s}}\Big|^2dx\Big]\\
&\leq &\sum_{k=m+1}^{\infty}\dfrac{64 N^2 e^{2\Lambda|t_1|}}{|\mu_k-2\Lambda|}\|T_{t_2-t_1}P^k-P^k\|^2\mathbb{E}\int_{{\cal O}}\int_{-\infty}^{t_1}e^{(\frac{1}{2}\mu_{m+1}-\Lambda)(t-\hat{s})}|F^k({\hat{s}},Y^N(\hat{s},\cdot))(x)|^2d\hat{s}dx\\
&&+\sum_{k=m+1}^{\infty}\dfrac{64 N^2 e^{2\Lambda|t_1|}}{|\mu_k+2\Lambda|}\|T_{t_2-t_1}P^k-P^k\|^2\mathbb{E}\int_{{\cal O}}\int_{-\infty}^{t_1}e^{(\frac{1}{2}\mu_{m+1}+\Lambda)(t-\hat{s})} |F^k({\hat{s}},Y^N(\hat{s},\cdot))(x)|^2d\hat{s}dx\\
&\leq& C N^2 e^{2\Lambda|t_1|}\bigg\{\sup_{k}\dfrac{|\mu_k|}{|\mu_k-2\Lambda|}\int_{-\infty}^{t_1}e^{(\frac{1}{2}\mu_{m+1}-\Lambda)(t-\hat{s})}\mathbb{E}\int_{{\cal O}}\sum_{k=m+1}^{\infty}|F^k({\hat{s}},Y^N(\hat{s},\cdot))(x)|^2dxd\hat{s}\\
&&+\sup_{k}\dfrac{|\mu_k|}{|\mu_k+2\Lambda|}\int_{-\infty}^{t_1}e^{(\frac{1}{2}\mu_{m+1}+\Lambda)(t-\hat{s})}\mathbb{E}\int_{{\cal O}}\sum_{k=m+1}^{\infty}|F^k({\hat{s}},Y^N(\hat{s},\cdot))(x)|^2dxd\hat{s}\bigg\}|t_2-t_1|\\
&\leq &CN^2\|F\|_{\infty}^2e^{2\Lambda|t_1|}\Big(\dfrac{1}{|\mu_{m+1}-2\Lambda|}+\dfrac{1}{|\mu_{m+1}+2\Lambda|}\Big)\sup_{k}\dfrac{|\mu_k|}{|\mu_k-2\Lambda|}|t_2-t_1|
\end{eqnarray*}
Totally we have
  \begin{eqnarray*}
T_1&\leq &CN^2\|F\|_{\infty}^2e^{2\Lambda|t_1|}\max\Big\{\sup_{k}e^{2\sigma_k^2  |t_2-t_1|},1\Big\}\Big(\sum_{k=1}^{\infty}\sigma_k^2 
\Big)\\
&&\quad\cdot\Big(\frac{1}{|\mu_{m+1}+2\Lambda|^2}+\frac{1}{|\mu_{m+1}-2\Lambda|^2}\Big)(|t_2-t_1|+|t_2-t_1|^2)
\\
&&+C N^2 e^{2\Lambda|t_1|}\|F\|^2_{\infty}\sup_{k}\Big(\dfrac{|\mu_{k}|^2}{|\mu_{k}-2\Lambda|^2}+\dfrac{|\mu_{k}|^2}{|\mu_{k}+2\Lambda|^2}\Big)|t_2-t_1|^2\\
&&+CN^2\|F\|_{\infty}^2e^{2\Lambda|t_1|}\Big(\dfrac{1}{|\mu_{m+1}-2\Lambda|}+\dfrac{1}{|\mu_{m+1}+2\Lambda|}\Big)\sup_{k}\dfrac{|\mu_k|}{|\mu_k-2\Lambda|}|t_2-t_1|.
\end{eqnarray*}
Similarly,
\begin{eqnarray*}
T_3&\leq &CN^2\|F\|_{\infty}^2e^{2\Lambda|t_1|}\max\Big\{\sup_{k}e^{2\sigma_k^2  |t_2-t_1|},1\Big\}\Big(\sum_{k=1}^{\infty}\sigma_k^2 
\Big)\\
&&\quad\cdot\Big(\frac{1}{|\mu_{m}+2\Lambda|^2}+\frac{1}{|\mu_{m}-2\Lambda|^2}\Big)(|t_2-t_1|+|t_2-t_1|^2)
\\
&&+C N^2 e^{2\Lambda|t_1|}\|F\|^2_{\infty}\sup_{k}\Big(\dfrac{|\mu_{k}|^2}{|\mu_{k}-2\Lambda|^2}+\dfrac{|\mu_{k}|^2}{|\mu_{k}+2\Lambda|^2}\Big)|t_2-t_1|^2\\
&&+CN^2\|F\|_{\infty}^2e^{2\Lambda|t_1|}\Big(\dfrac{1}{|\mu_{m}-2\Lambda|}+\dfrac{1}{|\mu_{m}+2\Lambda|}\Big)\sup_{k}\dfrac{|\mu_k|}{|\mu_k-2\Lambda|}|t_2-t_1|.
\end{eqnarray*}
  \item[(C)]  We show that ${\cal M}^N(Y^N)(t,\theta _{\pm\tau}\omega)={\cal M}^N(Y^N)(t\pm\tau,\omega)$:
  \begin{eqnarray*}
  {\cal M}^N(Y^N)(t,\theta_{\tau}\omega)&=&\int_{-\infty}^t\sum_{k=m+1}^{\infty}\Phi^N_{t-{\hat{s}},\hat{s}+\tau}P^kF({\hat{s}},Y^N({\hat{s}},\theta_{\tau}\omega))d{\hat{s}} \\
  & &-\int^{+\infty}_t\sum_{k=1}^{m}\Phi^N_{t-\hat{s},\hat{s}+\tau}P^kF(\hat{s},Y^N(\hat{s},\theta_{\tau}\omega))d\hat{s}\\
  &=&\int_{-\infty}^t\sum_{k=m+1}^{\infty}\Phi^N_{(t+\tau)-({\hat{s}}+\tau),\hat{s}+\tau}P^kF({\hat{s}}+\tau,Y^N({\hat{s}}+\tau,\omega))d{\hat{s}} \\
    & &-\int^{+\infty}_t\sum_{k=1}^{m}\Phi^N_{(t+\tau)-({\hat{s}}+\tau),\hat{s}+\tau}P^kF({\hat{s}}+\tau,Y^N({\hat{s}}+\tau,\omega))d{\hat{s}}\\
      &=&\int_{-\infty}^{t+\tau}\sum_{k=m+1}^{\infty}\Phi^N_{(t+\tau)-\hat{h},\hat{h}}P^kF(\hat{h},Y^N(\hat{h},\omega))d\hat{h} \\
        & &-\int^{+\infty}_{t+\tau}\sum_{k=1}^{m}\Phi^N_{(t+\tau)-\hat{h},\hat{h}}P^kF(\hat{h},Y^N(\hat{h},\omega))d\hat{h}\\
        &=&{\cal M}^N(Y^N)(t+\tau,\omega).
        \end{eqnarray*}
\end{enumerate}
Thus we have showed that ${\cal M}^N$ maps from $C^{\Lambda}_{\tau}(\mathbb{R}, L^2(\Omega\times {\cal O}))$ to itself.\\
{\it \textbf{Step 2}}: We need to prove that ${\cal M}^N$ mapping from $C^{\Lambda}_{\tau}(\mathbb{R}, L^2(\Omega\times {\cal O}))$ to $L_{\Lambda}^{\infty}(\mathbb{R}, L^2(\Omega, H_0^1({\cal O})))$.  
In fact for each $t$ and $\omega$ fixed, ${\cal M}^N(Y^N)(t,\omega,x)$ can be expressed as,
\begin{equation}\label{NJ520}
{\cal M}^N(Y^N)(t,\omega,x)=\sum_{i=1}^{\infty}\int_{{\cal O}}{\cal M}^N(Y^N)(t,\omega,y)\phi_i(y)dy\phi_i(x),
\end{equation} 
where $\phi_k\in H^1_0({\cal O})$, and we have
\begin{eqnarray*}
\nabla_x {\cal M}^N(Y^N)(t,\omega,x)&
=&\sum_{k=m+1}^{\infty}\int_{{\cal O}}\int_{-\infty}^t\Phi^N_{t-\hat{s},\hat{s}}P^kF(\hat{s},Y^N(\hat{s},\omega))(y)d\hat{s}\phi_k(y)dy\nabla_x\phi_k(x)\\
&&+\sum_{k=1}^{m}\int_{{\cal O}}\int^{+\infty}_t\Phi^N_{t-\hat{s},\hat{s}}P^kF(\hat{s},Y^N(\hat{s},\omega))(y)d\hat{s}\phi_k(y)dy\nabla_x\phi_k(x).\\
\end{eqnarray*}
Then we get
\begin{eqnarray*}
&&e^{-2\Lambda|t|}\E\int_{{\cal O}}|\nabla_x {\cal M}^N(Y^N)(t,\cdot,x)|^2dx\\
&\leq & 2e^{-2\Lambda|t|}\E\int_{{\cal O}}\Big|\sum_{k=m+1}^{\infty}\int_{{\cal O}}\int_{-\infty}^t\Phi^N_{t-\hat{s},\hat{s}}P^kF(\hat{s},Y^N(\hat{s},\cdot))(y)d\hat{s}\phi_k(y)dy\nabla_x\phi_k(x)\Big|^2dx\\
&&+2e^{-2\Lambda|t|}\E\int_{{\cal O}}\Big|\sum_{k=1}^{m}\int_{{\cal O}}\int^{\infty}_t\Phi^N_{t-\hat{s},\hat{s}}P^kF(\hat{s},Y^N(\hat{s},\cdot))(y)d\hat{s}\phi_k(y)dy\nabla_x\phi_k(x)\Big|^2dx\\
&=:&L_1+L_2.
\end{eqnarray*}
By Cauchy-Schwarz inequality, (\ref{eqn1.6}) and (\ref{PPk2}) we have
\begin{eqnarray*}
L_1&\leq & 2N^2e^{-2\Lambda|t|}\E\sum_{k,j=m+1}^{\infty}\Big(\int_{{\cal O}}|\nabla_x\phi_k(x)|^2dx\int_{{\cal O}}|\nabla_x\phi_j(x)|^2dx\Big)^{\frac{1}{2}} \\
&&\cdot\int_{-\infty}^t\int_{{\cal O}}e^{\frac{1}{2}\mu_k(t-\hat{s})}e^{\Lambda|\hat{s}|}|\phi_k(y)||F^k(\hat{s},Y^N(\hat{s},\cdot))(y)|dyd\hat{s}\\
&&\cdot\int_{-\infty}^t\int_{{\cal O}}e^{\frac{1}{2}\mu_j(t-\hat{s})}e^{\Lambda|\hat{s}|}|\phi_j(y)||F^j(\hat{s},Y^N(\hat{s},\cdot))(y)|dyd\hat{s}\\
&\leq & CN^2e^{-2\Lambda|t|}\E\Big[\sum_{k=m+1}^{\infty}\Big(\int_{-\infty}^t\int_{{\cal O}}e^{\frac{1}{2}\mu_k(t-\hat{s})}e^{\Lambda|\hat{s}|}\sqrt{|\mu_k|}|\phi_k(y)||F^k(\hat{s},Y^N(\hat{s},\cdot))(y)|dyd\hat{s}\Big)^{2} \\
&&\cdot\sum_{j=m+1}^{\infty}\Big(\int_{-\infty}^t\int_{{\cal O}}e^{\frac{1}{2}\mu_j(t-\hat{s})}e^{\Lambda|\hat{s}|}\sqrt{|\mu_j|}|\phi_j(y)||F^j(\hat{s},Y^N(\hat{s},\cdot))(y)|dyd\hat{s}\Big)^{2}\Big]^{\frac{1}{2}}\\
&\leq & CN^2e^{-2\Lambda|t|}\E\Big[\sum_{k=m+1}^{\infty}\Big(\int_{-\infty}^t\int_{{\cal O}}e^{\mu_k(t-\hat{s})}|\mu_k||\phi_k(y)|^2dyd\hat{s} \\
& & \cdot \int_{-\infty}^t\int_{{\cal O}}e^{\frac{1}{2}\mu_{m+1}(t-\hat{s})}e^{2\Lambda|\hat{s}|}|F^k(\hat{s},Y^N(\hat{s},\cdot))(y)|^2dyd\hat{s}\Big) \\
&&\cdot\sum_{j=m+1}^{\infty}\Big(\int_{-\infty}^t\int_{{\cal O}}e^{\mu_j(t-\hat{s})}|\mu_j||\phi_j(y)|^2dyd\hat{s}\\
&&\cdot \int_{-\infty}^t\int_{{\cal O}}e^{\frac{1}{2}\mu_{m+1}(t-\hat{s})}e^{2\Lambda|\hat{s}|}|F^j(\hat{s},Y^N(\hat{s},\cdot))(y)|^2dyd\hat{s}\Big)\Big]^{\frac{1}{2}}\\
&\leq & CN^2\|F\|^2_{\infty}\big(\frac{1}{|\mu_{m+1}-2\Lambda|}+\frac{1}{|\mu_{m+1}+2\Lambda|}\big)<\infty.
\end{eqnarray*}
Similarly,
$$L_2\leq CN^2\|F\|^2_{\infty}\big(\frac{1}{|\mu_{m}-2\Lambda|}+\frac{1}{|\mu_{m}+2\Lambda|}\big)<\infty.$$
{\it \textbf{Step 3}}: We now check the continuity of the map ${\cal M}^N$ in $C^{\Lambda}_{\tau}(\mathbb{R}, L^2(\Omega\times {\cal O}))$. Consider $Y_1^N,Y_2^N\in
C^{\Lambda}_{\tau}(\mathbb{R}, L^2(\Omega\times {\cal O}))$,  and $t\in [j\tau,(j+1)\tau)$, $j\in \mathbb{Z}$, then
\begin{eqnarray*}
&&e^{-2\Lambda|t|}\mathbb{E}\int_{{\cal O}}|{\cal M}^N(Y^N_1)(t,\cdot, x)-{\cal M}(Y^N_2)(t,\cdot, x)|^2dx\\
&\leq &2e^{-2\Lambda|t|}\mathbb{E}\int_{{\cal O}}\Big|\int_{-\infty}^t\sum_{k=m+1}^{\infty}\Phi^N_{t-\hat{s},\hat{s}}P^kF(\hat{s},Y_1^N)(x)d\hat{s}-\int_{-\infty}^t\sum_{k=m+1}^{\infty}\Phi^N_{t-\hat{s},\hat{s}}P^kF(\hat{s},Y_2^N)(x)d\hat{s}\Big|^2dx\\
&&+2e^{-2\Lambda|t|}\mathbb{E}\int_{{\cal O}}\Big|\int^{+\infty}_t\sum_{k=1}^{m}\Phi^N_{t-\hat{s},\hat{s}}P^kF(\hat{s},Y_1^N)(x)d\hat{s}-\int^{+\infty}_t\sum_{k=1}^{m}\Phi^N_{t-\hat{s},\hat{s}}P^kF(\hat{s},Y_2^N)(x)d\hat{s}\Big|^2dx,\\
&=:&U_1 +U_2.
\end{eqnarray*}
First
\begin{eqnarray*}
U_1&\leq &2e^{-2\Lambda|t|}\mathbb{E}\sum_{k=m+1}^{\infty}\int_{{\cal O}}\Big(\int_{-\infty}^t\|\Phi^N_{t-\hat{s},\hat{s}}P^k\||F^k(\hat{s},Y_1^N)(x)-F^k(\hat{s},Y_2^N)(x)|d\hat{s}\Big)^2dx\\
&\leq &2N^2e^{-2\Lambda|t|}\mathbb{E}\sum_{k=m+1}^{\infty}\int_{{\cal O}}\Big(\int_{-\infty}^te^{\frac{1}{2}\mu_{k}(t-\hat{s})}e^{\Lambda|\hat{s}|}|F^k(\hat{s},Y_1^N)(x)-F^k(\hat{s},Y_2^N)(x)|d\hat{s}\Big)^2dx\\
&\leq & 4N^2\mathbb{E}\sum_{k=m+1}^{\infty}\int_{{\cal O}}\Big(\int_{-\infty}^te^{(\frac{1}{2}\mu_{k}-\Lambda)(t-\hat{s})}|F^k(\hat{s},Y_1^N)(x)-F^k(\hat{s},Y_2^N)(x)|d\hat{s}\Big)^2dx\\
&&+4N^2\mathbb{E}\sum_{k=m+1}^{\infty}\int_{{\cal O}}\Big(\int_{-\infty}^te^{(\frac{1}{2}\mu_{k}+\Lambda)(t-\hat{s})}|F^k(\hat{s},Y_1^N)(x)-F^k(\hat{s},Y_2^N)(x)|d\hat{s}\Big)^2dx\\
&\leq &\frac{8N^2}{|\mu_{m+1}-2\Lambda|}\int_{-\infty}^te^{(\frac{1}{2}\mu_{m+1}-\Lambda)(t-\hat{s})}\mathbb{E}\sum_{k=m+1}^{\infty}\int_{{\cal O}}|F^k(\hat{s},Y_1^N)(x)-F^k(\hat{s},Y_2^N)(x)|^2dxd\hat{s}\\
&&+\frac{8N^2}{|\mu_{m+1}-2\Lambda|}\int_{-\infty}^te^{(\frac{1}{2}\mu_{m+1}+\Lambda)(t-\hat{s})}\mathbb{E}\sum_{k=m+1}^{\infty}\int_{{\cal O}}|F^k(\hat{s},Y_1^N)(x)-F^k(\hat{s},Y_2^N)(x)|^2dxd\hat{s}.
\end{eqnarray*}
Since the term $\mathbb{E}\sum_{k=m+1}^{\infty}\int_{{\cal O}}|F^k(\hat{s},Y_1^N)(x)-F^k(\hat{s},Y_2^N)(x)|^2dx$ is periodic in time, we can approach the following boundedness by the similar arguments in \cite{feng-wu-zhao}, i.e.,
\begin{eqnarray*}
U_1\leq e^{2\Lambda\tau}\Big(\frac{16N^2\|\nabla  F\|_{\infty}^2}{|\mu_{m+1}+2\Lambda|^2}+\frac{16N^2\|\nabla  F\|_{\infty}^2}{|\mu_{m+1}-2\Lambda|^2}\Big)\sup_{s\in \mathbb{R}}e^{-2\Lambda|\hat{s}|}\mathbb{E}\int_{{\cal O}}|Y_1^N(\hat{s},\cdot,x)-Y_2^N(\hat{s},\cdot,x)|^2dx,
\end{eqnarray*}
where
$$||\nabla  F||_\infty:=\sup_{t\in\mathbb{R}, x\in \mathbb{R}}|\nabla  F(t,x)|^2.$$
Analogously,
$$U_2\leq e^{2\Lambda\tau}\Big(\frac{16N^2\|\nabla  F\|_{\infty}^2}{|\mu_{m}+2\Lambda|^2}+\frac{16N^2\|\nabla  F\|_{\infty}^2}{|\mu_{m}-2\Lambda|^2}\Big)\sup_{s\in \mathbb{R}}e^{-2\Lambda|\hat{s}|}\mathbb{E}\int_{{\cal O}}|Y_1^N(\hat{s},\cdot,x)-Y_2^N(\hat{s},\cdot,x)|^2dx.$$
Thus the claim that ${\cal M}^N: C_{\tau}(\mathbb{R}, L^2(\Omega\times {\cal O}))\to
C_{\tau}(\mathbb{R}, L^2(\Omega\times {\cal O}))$ is a continuous map has been asserted.  
\end{proof}

\begin{rmk}
In Step 1 (B), we make complicated computation via (\ref{eqn:PhiT}) because splitting the terms in (\ref{eqn:PhiT}) together with estimation in (\ref{eqn:T_1,1}) avoids stronger assumption like $\sum_{k=1}^\infty \frac{1}{|\mu_k+2\Lambda|}<\infty$.
\end{rmk}
Before introducing the subset of $C^{\Lambda}_{\tau}(\mathbb{R}, L^2(\Omega\times{\cal O}))$ required in Schauder's fixed point Theorem, some calculation of the Malliavin derivatives of $\Phi^NP^j$ for each $j\in \mathbb{N}$ are needed.
\begin{lem} Given $\Phi^NP^j$ defined in (\ref{PHI+B1}) and (\ref{PHI-B1}), then for any $v\in L^2({\cal O})$,  $j\geq m+1$ and $t\geq 0$, the Malliavin derivative of ${\cal D}_r^j\Phi^NP^k$ is given by
\begin{eqnarray}\label{441}
&&{\cal D}^j_r\Phi^N(t,\theta_{\hat{s}}\omega)P^j(v)(x)\nonumber\\
&=&\sigma_j \chi_{[\hat{s}, t+\hat{s}]}(r)\min\left\{1,  \frac{Ne^{\frac{1}{2}\mu_{j}t}e^{\Lambda|s|}}{\| \Phi(t,\theta_{\hat{s}}\omega)P^{j}\|}\right\}\Phi(t,\theta_{\hat{s}}\omega)P^j(v)(x)\nonumber\\
&&-\sigma_j \chi_{[\hat{s}, t+\hat{s}]}(r)\chi_{\{Ne^{\frac{1}{2}\mu_{j}t}e^{\Lambda|s|}<\| \Phi(t,\theta_{\hat{s}}\omega)P^{j}\|\}}(\omega) \frac{Ne^{\frac{1}{2}\mu_{j}t}e^{\Lambda|s|}}{\| \Phi(t,\theta_{\hat{s}}\omega)P^{j}\|}\Phi(t,\theta_{\hat{s}}\omega)P^j(v)(x),\nonumber\\
\end{eqnarray}
and when $k\neq j$, $t\geq 0$, 
\begin{eqnarray}\label{441'}
{\cal D}^j_r\Phi^N(t,\theta_{\hat{s}}\omega)P^k(v)(x)=0.
\end{eqnarray}
And for $j\leq m$, $t\leq 0$, 
\begin{eqnarray}\label{442}
&&{\cal D}^j_r\Phi^N(t,\theta_{\hat{s}}\omega)P^j(v)(x)\nonumber\\
&=&-\sigma_j \chi_{[\hat{s}, t+\hat{s}]}(r)\min\Big\{1,  \frac{Ne^{\frac{1}{2}\mu_{j}t}e^{\Lambda|s|}}{\| \Phi(t,\theta_{\hat{s}}\omega)P^{j}\|}\Big\}\Phi(t,\theta_{\hat{s}}\omega)P^j(v)(x)\nonumber\\
&&+\sigma_j \chi_{[\hat{s}, t+\hat{s}]}(r)\chi_{\{Ne^{\frac{1}{2}\mu_{j}t}e^{\Lambda|s|}<\| \Phi(t,\theta_{\hat{s}}\omega)P^{j}\|\}}(\omega)\frac{Ne^{\frac{1}{2}\mu_{j}t}e^{\Lambda|s|}}{\| \Phi(t,\theta_{\hat{s}}\omega)P^{j}\|}\Phi(t,\theta_{\hat{s}}\omega)P^j(v)(x), \nonumber\\
\end{eqnarray}
and when $k\neq j$, $t\leq 0$,
\begin{eqnarray}\label{442'}
{\cal D}^j_r\Phi^N(t,\theta_{\hat{s}}\omega)P^k(v)(x)=0.
\end{eqnarray}
Moreover, we have the following estimate
\begin{equation}\label{44444}
\|{\cal D}^j_{r}\Phi^N(t,\theta_{\hat{s}}\omega)P^j\|\leq 2\sigma_j Ne^{\frac{1}{2}\mu_jt}e^{\Lambda|\hat{s}|}.
\end{equation} 
\end{lem}

\begin{proof}
 Firstly note that for any $v\in L^2({\cal O})$ we have for $j\geq m+1$ and $t\geq 0$,
\begin{eqnarray}\label{MD1}
{\cal D}^j_r\Phi(t,\theta_{\hat{s}}\omega)P^j(v)(x)&=&{\cal D}^j_re^{\mu_jt+\sigma_j (W^j_{\hat{s}+t}-W^j_{\hat{s}})}\langle\phi_j(\cdot), v(\cdot)\rangle \phi_j(x)\nonumber \\
&=&  \sigma_j \chi_{[\hat{s}, t+\hat{s}]}(r)e^{\mu_jt+\sigma_j (W^j_{\hat{s}+t}-W^j_{\hat{s}})}\langle\phi_j(\cdot), v(\cdot)\rangle \phi_j(x)\nonumber\\
&=&\sigma_j \chi_{[\hat{s}, t+\hat{s}]}(r)\Phi(t,\theta_{\hat{s}}\omega)P^j(v)(x),
\end{eqnarray}
and when $k\neq j$, 
\begin{eqnarray}\label{MD1'}
{\cal D}^j_r\Phi(t,\theta_{\hat{s}}\omega)P^k(v)(x)&=&{\cal D}^j_re^{\mu_kt+\sigma_k (W^k_{\hat{s}+t}-W^k_{\hat{s}})}\langle\phi_k(\cdot), v(\cdot)\rangle \phi_k(x)=0.
\end{eqnarray}
Moreover, 
\begin{eqnarray}\label{MD3}
{\cal D}^j_r\|\Phi(t,\theta_{\hat{s}}\omega)P^j\|={\cal D}^j_re^{\frac{1}{2}\mu_jt+\sigma_j (W^j_{\hat{s}+t}-W^j_{\hat{s}})}=\sigma_j \chi_{[\hat{s}, t+\hat{s}]}(r)\|\Phi(t,\theta_{\hat{s}}\omega)P^j\|,
\end{eqnarray}
and when $k\neq j$, 
\begin{eqnarray}\label{MD3'}
{\cal D}^j_r\|\Phi(t,\theta_{\hat{s}}\omega)P^k\|=0.
\end{eqnarray}
Analogously, when $j\leq m$ and $t\leq 0$,
\begin{equation}\label{MD2}
{\cal D}^j_r\Phi(t,\theta_{\hat{s}}\omega)P^j(v)(x)=-\sigma_j \chi_{[\hat{s}, t+\hat{s}]}(r)\Phi(t,\theta_{\hat{s}}\omega)P^j(v)(x),
\end{equation}
and
\begin{eqnarray}\label{MD4}
{\cal D}^j_r\|\Phi(t,\theta_{\hat{s}}\omega)P^j\|&=&-\sigma_j \chi_{[\hat{s}, t+\hat{s}]}(r)\|\Phi(t,\theta_{\hat{s}}\omega)P^j\|.
\end{eqnarray}
And when $k\neq j$, 
\begin{eqnarray}\label{MD2'}
{\cal D}^j_r\Phi(t,\theta_{\hat{s}}\omega)P^k(v)(x)=0,
\end{eqnarray}
and
\begin{eqnarray}\label{MD4'}
{\cal D}^j_r\|\Phi(t,\theta_{\hat{s}}\omega)P^k\|=0.
\end{eqnarray}
Then we are able to calculate the Malliavin derivatives of $\Phi^N$ by the chain rule using (\ref{MD1}) - (\ref{MD4'}): for $j\geq m+1$ and $t\geq 0$,
\begin{eqnarray*}
&&{\cal D}^j_r\Phi^N(t,\theta_{\hat{s}}\omega)P^j(v)(x)\nonumber\\
&=&
{\cal D}^j_r\left(\min\left\{1,  \frac{Ne^{\frac{1}{2}\mu_{j}t}e^{\Lambda|s|}}{\| \Phi(t,\theta_{\hat{s}}\omega)P^{j}\|}\right\}\Phi(t,\theta_{\hat{s}}\omega)P^{j}(v)(x)\right)\nonumber\\
&=&
\min\left\{1,  \frac{Ne^{\frac{1}{2}\mu_{j}t}e^{\Lambda|s|}}{\| \Phi(t,\theta_{\hat{s}}\omega)P^{j}\|}\right\}{\cal D}_r^j\big(\Phi(t,\theta_{\hat{s}}\omega)P^j(v)(x)\big)+{\cal D}_r^j\min\left\{1,  \frac{Ne^{\frac{1}{2}\mu_{j}t}e^{\Lambda|s|}}{\| \Phi(t,\theta_{\hat{s}}\omega)P^{j}\|}\right\}\Phi(t,\theta_{\hat{s}}\omega)P^j(v)(x)\nonumber\\
&=&\min\left\{1,  \frac{Ne^{\frac{1}{2}\mu_{j}t}e^{\Lambda|s|}}{\| \Phi(t,\theta_{\hat{s}}\omega)P^{j}\|}\right\}\sigma_j \chi_{[ \hat{s}, t+\hat{s}]}(r)\Phi(t,\theta_{\hat{s}}\omega)P^j(v)(x)\nonumber\\
&&-\chi_{\{Ne^{\frac{1}{2}\mu_{j}t}e^{\Lambda|s|}<\| \Phi(t,\theta_{\hat{s}}\omega)P^{j}\|\}}(\omega)\Phi(t,\theta_{\hat{s}}\omega)P^j(v)(x) \frac{Ne^{\frac{1}{2}\mu_{j}t}e^{\Lambda|s|}}{\| \Phi(t,\theta_{\hat{s}}\omega)P^{j}\|^2}{\cal D}_r^j\| \Phi(t,\theta_{\hat{s}}\omega)P^{j}\|\nonumber\\
&=&\sigma_j \chi_{[ \hat{s}, t+\hat{s}]}(r)\min\left\{1,  \frac{Ne^{\frac{1}{2}\mu_{j}t}e^{\Lambda|s|}}{\| \Phi(t,\theta_{\hat{s}}\omega)P^{j}\|}\right\}\Phi(t,\theta_{\hat{s}}\omega)P^j(v)(x)\nonumber\\
&&-\sigma_j \chi_{[ \hat{s}, t+\hat{s}]}(r)\chi_{\{Ne^{\frac{1}{2}\mu_{j}t}e^{\Lambda|s|}<\| \Phi(t,\theta_{\hat{s}}\omega)P^{j}\|\}}(\omega) \frac{Ne^{\frac{1}{2}\mu_{j}t}e^{\Lambda|s|}}{\| \Phi(t,\theta_{\hat{s}}\omega)P^{j}\|}\Phi(t,\theta_{\hat{s}}\omega)P^j(v)(x),
\end{eqnarray*}
and when $k\neq j$, it is obvious that
\begin{eqnarray*}
{\cal D}^j_r\Phi^N(t,\theta_{\hat{s}}\omega)P^k(v)(x)=0.
\end{eqnarray*}
Analogously, for $j\leq m$ with $t\leq 0$, we have
\begin{eqnarray*}
&&{\cal D}^j_r\Phi^N(t,\theta_{\hat{s}}\omega)P^j(v)(x)\nonumber\\
&=&-\sigma_j \chi_{[ \hat{s}, t+\hat{s}]}(r)\min\Big\{1,  \frac{Ne^{\frac{1}{2}\mu_{j}t}e^{\Lambda|s|}}{\| \Phi(t,\theta_{\hat{s}}\omega)P^{j}\|}\Big\}\Phi(t,\theta_{\hat{s}}\omega)P^j(v)(x)\nonumber\\
&&+\sigma_j \chi_{[ \hat{s}, t+\hat{s}]}(r)\chi_{\{Ne^{\frac{1}{2}\mu_{j}t}e^{\Lambda|s|}<\| \Phi(t,\theta_{\hat{s}}\omega)P^{j}\|\}}(\omega)\frac{Ne^{\frac{1}{2}\mu_{j}t}e^{\Lambda|s|}}{\| \Phi(t,\theta_{\hat{s}}\omega)P^{j}\|}\Phi(t,\theta_{\hat{s}}\omega)P^j(v)(x),
\end{eqnarray*}
and when $k\neq j$, it is obvious that
\begin{eqnarray*}
{\cal D}^j_r\Phi^N(t,\theta_{\hat{s}}\omega)P^k(v)(x)=0.
\end{eqnarray*}
Finally from  (\ref{441}) and (\ref{442}), it is easy to obtain that when $j\geq m+1$ and $t\geq 0$ or $j\leq m$ and $t\leq 0$
\begin{equation*}
\|{\cal D}^j_{r}\Phi^N(t,\theta_{\hat{s}}\omega)P^j\|\leq 2\sigma_j Ne^{\frac{1}{2}\mu_jt}e^{\Lambda|\hat{s}|}.
\end{equation*}  
\end{proof}

Next introduce a subset of $C^{\Lambda}_{\tau}(\mathbb{R}, L^2(\Omega\times{\cal O}))$ as follows,
\begin{eqnarray*}
C^{\Lambda,N}_{\tau,\rho}(\mathbb{R},L^2({\cal O},{\cal D}^{1,2}))&:=\bigg\{& f\in C^{\Lambda}_{\tau}(\mathbb{R}, L^2(\Omega\times{\cal O})):\ f|_{[0,\tau)}\in C([0,\tau),L^2({\cal O},{\cal D}^{1,2})), \\
&&  \mbox{and}\ {\forall} t\in [0,\tau),\ e^{-2\Lambda|t|}\sum_{j=1}^{\infty}\int_{\mathbb{R}}\mathbb{E}\int_{{\cal O}}|{\cal D}_{r}^jf(t,\cdot,x)|^2dxdr\leq \rho^N(t),\\
&&\sup\limits_{\stackrel{t\in [0,\tau)}{\delta\in \mathbb{R}}} e^{-2\Lambda|t|}\sum_{j=1}^{\infty}\dfrac{1}{|\delta|}\int_{\mathbb{R}}\mathbb{E}\int_{{\cal O}}|{\cal D}_{r+\delta}^jf(t,\cdot,x)-{\cal D}_{r}^jf(t,\cdot,x)|^2dxdr< \infty  \nonumber\bigg\}.
\end{eqnarray*}
Here
\begin{equation}\label{NBJIHU}
\rho^N(t):=K^N_1\int_{0}^{\tau}e^{-\frac{1}{2}\mu|t-\hat{s}|}\rho^N(\hat{s})d\hat{s}+K^N_2,
\end{equation}
where 
$$K^N_1:=12N^2\|\nabla F\|^2_{\infty}e^{2\Lambda \tau}\Big(\dfrac{\sum_{i=-1}^{\infty}e^{-\frac{1}{2}\mu_{m}i\tau}}{|\mu_{m}-4\Lambda|}+\dfrac{\sum_{i=-1}^{\infty}e^{-\frac{1}{2}\mu_{m}i\tau}}{|\mu_{m}+4\Lambda|}+\dfrac{\sum_{i=-1}^{\infty}e^{\frac{1}{2}\mu_{m+1}i\tau}}{|\mu_{m+1}-4\Lambda|}+\dfrac{\sum_{i=-1}^{\infty}e^{\frac{1}{2}\mu_{m+1}i\tau}}{|\mu_{m+1}+4\Lambda|}\Big),$$
$$K^N_2:=96\|F\|^2_{\infty}\sum_{j=1}^{\infty}\sigma_j^2 \Big(\dfrac{1}{|\mu_{m}+2\Lambda|^3}+\dfrac{1}{|\mu_{m}-2\Lambda|^3}+\dfrac{1}{|\mu_{m+1}+2\Lambda|^3}+\dfrac{1}{|\mu_{m+1}-2\Lambda|^3}\Big).$$
We can prove that
\begin{lem}\label{LEMMA4C4}Under the conditions of Theorem \ref{THMC4M}, we have
$${\cal M}^N:C^{\Lambda,N}_{\tau,\rho}(\mathbb{R},L^2({\cal O},{\cal D}^{1,2})) \subset C^{\Lambda,N}_{\tau,\rho}(\mathbb{R},L^2({\cal O},{\cal D}^{1,2})),$$
and ${\cal M}^N(C^{\Lambda,N}_{\tau,\rho}(\mathbb{R},L^2({\cal O},{\cal D}^{1,2})))|_{[0,\tau)}$ is relatively compact in $C_{\tau}([0,\tau),L^2(\Omega\times{\cal O}))$.
\end{lem}
\begin{proof} {\it \textbf{Step 1}}:
Now we show that for any $t\in [0,\tau)$,
$$\sum_{j=1}^{\infty}e^{-2\Lambda|t|}\mathbb{E}\int_{\mathbb{R}}\int_{{\cal O}}|{\mathcal{D}}^j_{r}{\cal M}^N(Y^N)(t,\cdot,x)|^2dxdr\leq \rho^N(t).$$
By the chain rule and (\ref{441}) - (\ref{442'}), it is easy to write down the  Malliavin derivative of ${\cal M}^N(Y^N)(t,\omega,x)$ with respect to the $j$th Brownian motion for $j\geq m+1$, 
\begin{eqnarray}\label{DM}
{\cal D}^j_{r}{\cal M}^N(Y^N)(t,\omega,x)&=&\int_{-\infty}^r\chi_{\{r\leq t\}}(r){\cal D}^j_{r}(\Phi^N_{t-\hat{s},\hat{s}}P^j)F(\hat{s}, Y^N(\hat{s},\omega))(x)d\hat{s}\nonumber\\
&&+\int_{-\infty}^t\sum_{k=m+1}^{\infty}\Phi^N_{t-\hat{s},\hat{s}}P^k\nabla F(\hat{s}, Y^N(\hat{s},\omega)){\cal D}^j_{r} Y^N(\hat{s},\omega,x)d\hat{s} \nonumber\\
&&-\int^{+\infty}_t\sum_{k=1}^{m}\Phi^N_{t-\hat{s},\hat{s}}P^k\nabla F(\hat{s}, Y^N(\hat{s},\omega)){\cal D}^j_{r} Y^N(\hat{s},\omega,x)d\hat{s}.
\end{eqnarray}
Then the following $L^2$-estimation follows:
\begin{eqnarray*}
&&e^{-2\Lambda|t|}\int_{\mathbb{R}}\mathbb{E}\int_{{\cal O}}|{\cal D}^j_{r}{\cal M}^N(Y^N)(t,\cdot,x)|^2dxdr\\
&=&e^{-2\Lambda|t|}\Big(\int_{-\infty}^t+\int^{\infty}_t\Big)\mathbb{E}\int_{{\cal O}}|{\cal D}^j_{r}{\cal M}^N(Y^N)(t,\cdot,x)|^2dxdr\\
&\leq & 3e^{-2\Lambda|t|}\int_{-\infty}^t\mathbb{E}\int_{{\cal O}}\Big|\int_{-\infty}^{r}{\cal D}_{r}^j(\Phi^N_{t-\hat{s},\hat{s}}P^j)F(\hat{s}, Y^N(\hat{s},\cdot))(x)d\hat{s}\Big|^2dxdr\\
&&+3e^{-2\Lambda|t|}\int_{\mathbb{R}}\mathbb{E}\sum_{k=m+1}^{\infty}\int_{{\cal O}}\Big|\int_{-\infty}^t\Phi^N_{t-\hat{s},\hat{s}}P^k\nabla F(\hat{s}, Y^N(\hat{s},\cdot,x)){\cal D}^j_{r} Y^N(\hat{s},\cdot,x)d\hat{s}\Big|^2dxdr\\
&&+3e^{-2\Lambda|t|}\int_{\mathbb{R}}\mathbb{E}\sum_{k=1}^{m}\int_{{\cal O}}\Big|\int^{+\infty}_t\Phi^N_{t-\hat{s},\hat{s}}P^k\nabla F(\hat{s}, Y^N(\hat{s},\omega,x)){\cal D}^j_{r} Y^N(\hat{s},\cdot,x)d\hat{s}\Big|^2dxdr\\
&&=:\sum_{i=1,2,3}L^j_i.
\end{eqnarray*}

Firstly, by (\ref{44444}) and trick in (\ref{eqn:trick}) we have that
\begin{eqnarray*}
L^j_1&\leq & 3\|F\|^2_{\infty}e^{-2\Lambda|t|}\int_{-\infty}^t\mathbb{E}\Big(\int_{-\infty}^{r}\|{\cal D}^j_{r}\Phi^N_{t-\hat{s},\hat{s}}P^j\|d\hat{s}\Big)^2dr\\
&\leq &12N^2\|F\|^2_{\infty}\sigma_j^2 e^{-2\Lambda|t|}\int_{-\infty}^t\Big(\int_{-\infty}^{r}e^{\frac{1}{2}\mu_j(t-\hat{s})}e^{\Lambda|\hat{s}|}d\hat{s}\Big)^2dr\\
&\leq & 96N^2\|F\|^2_{\infty}\left(\frac{\sigma_j^2 }{|\mu_j-2\Lambda|^3}+\frac{\sigma_j^2 }{|\mu_j+2\Lambda|^3}\right)\\
&\leq &96N^2\|F\|^2_{\infty}\left(\frac{\sigma_j^2 }{|\mu_{m+1}-2\Lambda|^3}+\frac{\sigma_j^2 }{|\mu_{m+1}+2\Lambda|^3}\right).
\end{eqnarray*}
Secondly, through the norm preserving property in Lemma \ref{Preserving},
\begin{eqnarray*}
L^j_2&=&3e^{-2\Lambda|t|}\int_{\mathbb{R}}\mathbb{E}\sum_{k=m+1}^{\infty}\int_{{\cal O}}\Big|\int_{-\infty}^t\Phi^N_{t-\hat{s},\hat{s}}P^k\nabla F(\hat{s}, Y^N)({\cal D}^j_{r} Y^N(\hat{s},\cdot,x))^{(k)}d\hat{s}\Big|^2dxdr\\
&\leq & 3e^{-2\Lambda|t|}N^2\int_{\mathbb{R}} \mathbb{E}\sum_{k=m+1}^{\infty}\int_{{\cal O}}\Big(\int_{-\infty}^te^{\frac{1}{2}\mu_k(t-\hat{s})}e^{\Lambda|\hat{s}|}|\nabla F^k(\hat{s}, Y^N)||({\cal D}^j_{r} Y^N(\hat{s},\cdot,x))^{(k)}|d\hat{s}\Big)^2dxdr\\
&\leq & 6N^2\int_{\mathbb{R}} \mathbb{E}\sum_{k=m+1}^{\infty}\int_{{\cal O}}\Big(\int_{-\infty}^te^{(\frac{1}{2}\mu_{m+1}+\Lambda)(t-\hat{s})}|\nabla F^k(\hat{s}, Y^N)||({\cal D}^j_{r} Y^N(\hat{s},\cdot,x))^{(k)}|d\hat{s}\Big)^2dxdr\\
&& + 6N^2\int_{\mathbb{R}} \mathbb{E}\sum_{k=m+1}^{\infty}\int_{{\cal O}}\Big(\int_{-\infty}^te^{(\frac{1}{2}\mu_{m+1}-\Lambda)(t-\hat{s})}|\nabla F^k(\hat{s}, Y^N)||({\cal D}^j_{r} Y^N(\hat{s},\cdot,x))^{(k)}|^2d\hat{s}\Big)^2dxdr\\
&\leq & \frac{12N^2}{|\mu_{m+1}+4\Lambda|}\int_{\mathbb{R}} \mathbb{E}\int_{{\cal O}}\int_{-\infty}^te^{\frac{1}{2}\mu_{m+1}(t-\hat{s})}\sum_{k=m+1}^{\infty}|\nabla F^k(\hat{s}, Y^N)|^2|({\cal D}^j_{r} Y^N(\hat{s},\cdot,x))^{(k)}|^2d\hat{s}dxdr\\
&& + \frac{12N^2}{|\mu_{m+1}-4\Lambda|}\int_{\mathbb{R}}\mathbb{E} \int_{{\cal O}}\int_{-\infty}^te^{\frac{1}{2}\mu_{m+1}(t-\hat{s})}\sum_{k=m+1}^{\infty}|\nabla F^k(\hat{s}, Y^N)|^2|({\cal D}^j_{r} Y^N(\hat{s},\cdot,x))^{(k)}|^2d\hat{s}dxdr\\
&=&\Big(\frac{12N^2\|\nabla F\|^2_{\infty}}{|\mu_{m+1}+4\Lambda|}+\frac{12N^2\|\nabla F\|^2_{\infty}}{|\mu_{m+1}-4\Lambda|}\Big)\\
&&\cdot\bigg\{\sum_{i=0}^{\infty}e^{\frac{1}{2}\mu_{m+1}i\tau}\int_{0}^{t}e^{-2\Lambda|\hat{s}|}\int_{\mathbb{R}}\int_{{\cal O}}e^{\frac{1}{2}\mu_{m+1}(t-\hat{s})}\mathbb{E}|{\cal D}^j_{r}Y^N(\hat{s},\theta_{-i\tau}\cdot, x)|^2 dx drd\hat{s}\\
&&\hspace{0.4cm}+\sum_{i=0}^{\infty}e^{\frac{1}{2}\mu_{m+1}i\tau}\int_{t}^{\tau}e^{-2\Lambda|\hat{s}|}\int_{\mathbb{R}}\int_{{\cal O}}e^{\frac{1}{2}\mu_{m+1}(t+\tau-\hat{s})}\mathbb{E}|{\cal D}^j_{r}Y^N(\hat{s},\theta_{-(i+1)\tau}\cdot, x)|^2 dx drd\hat{s}\bigg\}\\
&\leq&e^{(2\Lambda-\frac{1}{2}\mu_{m+1})\tau}\Big(\frac{12N^2\|\nabla F\|^2_{\infty}}{|\mu_{m+1}+4\Lambda|}+\frac{12N^2\|\nabla F\|^2_{\infty}}{|\mu_{m+1}-4\Lambda|}\Big)\Big(\sum_{i=0}^{\infty}e^{\frac{1}{2}\mu_{m+1}i\tau}\Big)\\
&&\cdot\int_{0}^{\tau}e^{\frac{1}{2}\mu_{m+1}|t-\hat{s}|}e^{-2\Lambda|\hat{s}|}\int_{\mathbb{R}}\int_{{\cal O}}\mathbb{E}|{\cal D}^j_{r}Y^N(\hat{s},\cdot, x)|^2 dx drd\hat{s},
\end{eqnarray*}
where 
$$({\cal D}^j_{r} Y^N(\hat{s},\omega,x))^{(k)}:={\cal D}^j_{r}\Big( \langle Y^N(\hat{s},\omega), \phi_k\rangle \phi_k(x)\Big).$$
Similarly
\begin{eqnarray*}
L^j_3&\leq &e^{(2\Lambda+\frac{1}{2}\mu_{m})\tau}\Big(\frac{12N^2\|\nabla F\|^2_{\infty}}{|\mu_{m}+4\Lambda|}+\frac{12N^2\|\nabla F\|^2_{\infty}}{|\mu_{m}-4\Lambda|}\Big)\Big(\sum_{i=0}^{\infty}e^{-\frac{1}{2}\mu_{m}i\tau}\Big)\\
&&\cdot\int_{0}^{\tau}e^{\frac{1}{2}\mu|t-\hat{s}|}e^{-2\Lambda|\hat{s}|}\int_{\mathbb{R}}\int_{{\cal O}}\mathbb{E}|{\cal D}^j_{r}Y^N(\hat{s},\cdot, x)|^2 dx drd\hat{s}.
\end{eqnarray*}
Similar calculation can be applied to the case for $j\leq m$. Therefore we finally reach to
$$\sum_{j=1}^{\infty}L^j_1=96\|F\|^2_{\infty}\sum_{j=1}^{\infty}\sigma_j^2 \left(\frac{1}{|\mu_{m+1}+2\Lambda|^3}+\frac{1}{|\mu_{m+1}-2\Lambda|^3}+\frac{1}{|\mu_m-2\Lambda|^3}+\frac{1}{|\mu_m+2\Lambda|^3}\right)=: K_2^N,$$
and
\begin{eqnarray*}
\sum_{j=1}^{\infty}(L^j_2+L^j_3)&\leq &e^{(2\Lambda+\frac{1}{2}\hat{\mu})\tau}\sum_{k=m}^{m+1}\Big(\frac{24N^2\|\nabla F\|^2_{\infty}}{|\mu_k+4\Lambda|}+\frac{24N^2\|\nabla F\|^2_{\infty}}{|\mu_k-4\Lambda|}\Big) \Big(\sum_{i=0}^{\infty}e^{-\frac{1}{2}|\mu_{k}|i\tau}\Big)\\
&&\cdot\int_{0}^{\tau}e^{\frac{1}{2}\mu|t-\hat{s}|}e^{-2\Lambda|\hat{s}|}\sum_{j=1}^{\infty}\int_{\mathbb{R}}\int_{{\cal O}}\mathbb{E}|{\cal D}^j_{r}Y^N(\hat{s},\cdot, x)|^2 dx drd\hat{s}\\
&\leq & K^N_1 \int_{0}^{\tau}e^{\frac{1}{2}\mu|t-\hat{s}|}\rho^N(s)d\hat{s}.
\end{eqnarray*}

To sum up, we have verified the following estimation:
$$\sum_{j=1}^{\infty}e^{-2\Lambda|t|}\int_{\mathbb{R}}\mathbb{E}\int_{{\cal O}}|{\mathcal{D}}^j_{r}{\cal M}^N(Y^N)(t,\cdot,x)|^2dxdr\leq\rho^N(t).$$

Moreover, the solution $\rho^N(t)$ to equation (\ref{NBJIHU}) is continuous in $t$, so that for $Y^N\in C_{\tau,\rho}^{\Lambda,N}(\mathbb{R},L^2({\cal O},\mathcal{D}^{1,2}))$, there exists an integer $N_b$ such that for any $t\in [0,\tau)$,  
$$\sum_{j=1}^{\infty}e^{-2\Lambda|t|}\int_{\mathbb{R}}\mathbb{E}\int_{{\cal O}}|{\mathcal{D}}^j_{r}{\cal M}^N(Y^N)(t,\cdot,x)|^2dxdr\leq \rho^N(t)\leq N_b.$$

It remains to show that for any $\delta\in \mathbb{R}$,
$$\sup\limits_{\substack{t\in [0,\tau)}}\dfrac{e^{-2\Lambda |t|}}{|\delta|}\sum_{j=1}^{\infty}\int_{\mathbb{R}}\mathbb{E}\int_{{\cal O}}|{\cal D}_{r+\delta}^j\mathcal{M}^N(Y^N)(t,\cdot,x)-{\cal D}_{r}^j\mathcal{M}^N(Y^N)(t,\cdot,x)|^2dxdr< \infty.$$
The left hand side of the above can be separated into three integrals, 
\begin{eqnarray}\label{QSPDE}
&&\sup\limits_{\substack{t\in (0,\tau]}}\dfrac{e^{-2\Lambda |t|}}{|\delta|} \sum_{j=1}^{\infty}\int_{\mathbb{R}}\mathbb{E}\int_{{\cal O}}|{\cal D}_{r+\delta}^j{\cal M}^N(Y^N)(t,\cdot,x)-{\cal D}_{r}^j{\cal M}^N(Y^N)(t,\cdot,x)|^2dxdr\nonumber\\
&=& \sup\limits_{\substack{t\in (0,\tau]}}\dfrac{e^{-2\Lambda  |t|}}{|\delta|} \sum_{j=1}^{\infty} \int_{-\infty}^{t-\delta}\mathbb{E}\int_{{\cal O}}|{\cal D}_{r+\delta}^j{\cal M}^N(Y^N)(t,\cdot,x)-{\cal D}_{r}^j{\cal M}^N(Y^N)(t,\cdot,x)|^2dxdr\nonumber\\
&&+\sup\limits_{\substack{t\in (0,\tau]}}\dfrac{e^{-2\Lambda  |t|}}{|\delta|}  \sum_{j=1}^{\infty}\int_{t-\delta}^{t}\mathbb{E}\int_{{\cal O}}|{\cal D}_{r+\delta}^j{\cal M}^N(Y^N)(t,\cdot,x)-{\cal D}_{r}^j{\cal M}^N(Y^N)(t,\cdot,x)|^2dxdr\nonumber\\
&&+\sup\limits_{\substack{t\in (0,\tau]}}\dfrac{e^{-2\Lambda  |t|}}{|\delta|}\sum_{j=1}^{\infty} \int_{t}^{+\infty}\mathbb{E}\int_{{\cal O}}|{\cal D}_{r+\delta}^j{\cal M}^N(Y^N)(t,\cdot,x)-{\cal D}_{r}^j{\cal M}^N(Y^N)(t,\cdot,x)|^2dxdr.\nonumber\\
&=:&\hat{K}_1+\hat{K}_2+\hat{K}_3.
\end{eqnarray}
The above domain $\mathbb{R}$ partitioned into three sub-domains is based on the fact that the Malliavin derivative stays the same in each of three sub-domains. 

To consider $\hat{K}_1$, note that when $r\leq t-\delta$, by (\ref{eqn:M}), (\ref{441}) and (\ref{442'}) we can first compute the Malliavin derivative of $\mathcal{M}$ and substitute it back, thus have
\begin{eqnarray*}
\hat{K}_1&\leq &\dfrac{3e^{-2\Lambda |t|}}{|\delta|}   \sum_{j=1}^{\infty} \int_{-\infty}^{t-\delta}\mathbb{E}\int_{{\cal O}}\bigg\{\Big|\int_{-\infty}^t  \sum_{k=m+1}^{\infty}\Phi_{t-\hat{s},\hat{s}}^NP^k\nabla F(\hat{s},Y^N(\hat{s},\cdot,x))({\cal D}_{r+\delta}^j-{\cal D}_{r}^j)(Y^N)(t,\cdot,x)d\hat{s}\Big|^2\\
&&+\chi_{\{j\geq m+1\}}(j)\Big|\int_{-\infty}^{r+\delta} {\cal D}_{r+\delta}^j\Phi_{t-\hat{s},\hat{s}}^NP^jF(\hat{s},Y^N(\hat{s},\cdot,x))d\hat{s}-\int_{-\infty}^{r} {\cal D}_{r}^j\Phi_{t-\hat{s},\hat{s}}^NP^jF(\hat{s},Y^N(\hat{s},\cdot,x))d\hat{s}\Big|^2\\
&&+\Big|\int^{+\infty}_t  \sum_{k=1}^{m}\Phi_{t-\hat{s},\hat{s}}^NP^k\nabla F(\hat{s},Y^N(\hat{s},\cdot,x))({\cal D}_{r+\delta}^j-{\cal D}_{r}^j)(Y^N)(t,\cdot,x)d\hat{s}\Big|^2\bigg\}dxdr\\
&=&\sum_{i=1}^{3}Q_i.
\end{eqnarray*}
First note that $Q_1$ is bounded as follows,
\begin{eqnarray*}
Q_1&\leq &\dfrac{3N^2}{|\delta|}e^{-2\Lambda |t|}\sum_{k=m+1}^{\infty}\mathbb{E}\sum_{j=1}^{\infty}\int_{{\cal O}}\int_{\mathbb{R}}\Big(\int_{-\infty}^t e^{\Lambda |\hat{s}|}e^{\frac{1}{2}\mu_{k}(t-\hat{s})} \\
&&\quad \quad \cdot|\nabla F^k(\hat{s},Y^N(\hat{s},\cdot,x))||(({\cal D}_{r+\delta}^j-{\cal D}_{r}^j)(Y^N)(\hat{s},\cdot,x))^{(k)}|d\hat{s}\Big)^2dxdr\\
&\leq &\dfrac{12N^2}{|\delta|}\Big(\frac{1}{|\mu_{m+1}+4\Lambda|}+\frac{1}{|\mu_{m+1}-4\Lambda|}\Big)\\
&&\cdot\sum_{j=1}^{\infty}\int_{\mathbb{R}}\mathbb{E}\int_{{\cal O}}\int_{-\infty}^t e^{\frac{1}{2}\mu_{m+1}(t-\hat{s})}\sum_{k=m+1}^{\infty}|\nabla F^k(\hat{s},Y^N(\hat{s},\cdot,x))|^2 |(({\cal D}_{r+\delta}^j-{\cal D}_{r}^j)(Y^N)(\hat{s},\cdot))^{(k)}|^2d\hat{s}dxdr\\
&\leq &12N^2\|\nabla F\|^2_{\infty}\Big(\frac{1}{|\mu_{m+1}+4\Lambda|}+\frac{1}{|\mu_{m+1}-4\Lambda|}\Big)\\
&&\cdot\dfrac{1}{|\delta|}\sum_{j=1}^{\infty}\int_{-\infty}^t e^{\frac{1}{2}\mu_{m+1}(t-\hat{s})}\mathbb{E}\int_{\mathbb{R}}\int_{{\cal O}} |({\cal D}_{r+\delta}^j-{\cal D}_{r}^j)(Y^N)(\hat{s},\cdot,x)|^2dxdrd\hat{s}.
\end{eqnarray*}
From Lemma \ref{Preserving} we can see that the term $\mathbb{E}\int_{\mathbb{R}}\int_{{\cal O}} |({\cal D}_{r+\delta}^j-{\cal D}_{r}^j)(Y^N)(\hat{s},\cdot,x)|^2dxdr$ is periodic in time, thus one one period is needed to address the boundedness. This leads to the following estimate
\begin{eqnarray*}
Q_1&\leq &12N^2e^{2\Lambda \tau}\|\nabla F\|^2_{\infty}\Big(\frac{1}{|\mu_{m+1}+4\Lambda|}+\frac{1}{|\mu_{m+1}-4\Lambda|}\Big)\\
&&\cdot\sum_{j=1}^{\infty}\int_{-\infty}^t e^{\frac{1}{2}\mu_{m+1}(t-\hat{s})}\dfrac{e^{-2\Lambda|\hat{s}|}}{|\delta|}\mathbb{E}\int_{\mathbb{R}}\int_{{\cal O}} |({\cal D}_{r+\delta}^j-{\cal D}_{r}^j)(Y^N)(\hat{s},\cdot,x)|^2dxdrd\hat{s}\\
&\leq &24N^2\|\nabla F\|^2_{\infty}e^{2\Lambda \tau}\frac{1}{|\mu_{m+1}|}\Big(\frac{1}{|\mu_{m+1}+4\Lambda|}+\frac{1}{|\mu_{k}-4\Lambda|}\Big)\\
&&\sup_{\hat{s}\in [0,\tau)}\dfrac{e^{-2\Lambda|\hat{s}|}}{|\delta|}\sum_{j=1}^{\infty}\int_{\mathbb{R}}\mathbb{E}\int_{{\cal O}} |({\cal D}_{r+\delta}^j-{\cal D}_{r}^j)(Y^N)(\hat{s},\cdot,x)|^2dxdr<\infty.
\end{eqnarray*}
Analogously, 
\begin{eqnarray*}
Q_3&\leq &24e^{2\Lambda \tau}\frac{1}{|\mu_{m}|}\Big(\frac{N^2\|\nabla F\|^2_{\infty}}{|\mu_{m}+4\Lambda|}+\frac{N^2\|\nabla F\|^2_{\infty}}{|\mu_{m}-4\Lambda|}\Big) \\
&&\cdot\sup_{\hat{s}\in [0,\tau)}\dfrac{e^{-2\Lambda|\hat{s}|}}{|\delta|}\sum_{j=1}^{\infty}\int_{\mathbb{R}}\mathbb{E}\int_{{\cal O}} |({\cal D}_{r+\delta}^j-{\cal D}_{r}^j)(Y^N)(\hat{s},\cdot,x)|^2dxdr<\infty.
\end{eqnarray*}
Moreover by (\ref{441}) we have
\begin{eqnarray*}
Q_2&=&\dfrac{3e^{-2\Lambda |t|}}{|\delta|} \sum_{j=m+1}^{\infty} \int_{-\infty}^{t-\delta}\mathbb{E}\int_{{\cal O}}\Big|\int_{-\infty}^{r+\delta} {\cal D}_{r+\delta}^j\Phi_{t-\hat{s},\hat{s}}^NP^jF(\hat{s},Y^N(\hat{s},\cdot,x))d\hat{s}\\
&&\hspace{4.5cm}-\int_{-\infty}^{r} {\cal D}_{r}^j\Phi_{t-\hat{s},\hat{s}}^NP^jF(\hat{s},Y^N(\hat{s},\cdot,x))d\hat{s}\Big|^2dxdr\\
&\leq & \dfrac{3e^{-2\Lambda |t|}}{|\delta|} \sum_{j=1}^{\infty}\int_{-\infty}^{t-\delta}\mathbb{E}\int_{{\cal O}}\Big(\int_{r}^{r+\delta} \|{\cal D}_{r+\delta}^j\Phi_{t-\hat{s},\hat{s}}^NP^j\||F^j(\hat{s},Y^N(\hat{s},\cdot,x))|d\hat{s}\Big)^2dxdr\\
&\leq &\sigma_j^2  \dfrac{12N^2e^{-2\Lambda |t|}}{|\delta|} \sum_{j=1}^{\infty}\int_{-\infty}^{t-\delta}\mathbb{E}\int_{{\cal O}}\Big(\int_{r}^{r+\delta} e^{\frac{1}{2}\mu_{j}(t-\hat{s})}e^{\Lambda|\hat{s}|}|F^j(\hat{s},Y^N(\hat{s},\cdot,x))|d\hat{s}\Big)^2dxdr\\
&\leq & \sigma_j^2 \dfrac{24N^2\sigma_j^2 }{|\delta|} \sum_{j=1}^{\infty}\int_{-\infty}^{t-\delta}\mathbb{E}\int_{{\cal O}}\Big(\int_{r}^{r+\delta} e^{(\frac{1}{2}\mu_{j}-\Lambda)(t-\hat{s})}|F^j(\hat{s},Y^N(\hat{s},\cdot,x))|d\hat{s}\Big)^2dxdr\\
&&+\sigma_j^2 \dfrac{24N^2}{|\delta|} \sum_{j=1}^{\infty}\int_{-\infty}^{t-\delta}\mathbb{E}\int_{{\cal O}}\Big(\int_{r}^{r+\delta} e^{(\frac{1}{2}\mu_{j}+\Lambda)(t-\hat{s})}|F^j(\hat{s},Y^N(\hat{s},\cdot,x))|d\hat{s}\Big)^2dxdr\\
&\leq & \dfrac{24N^2\sigma_j^2 }{|\delta|} \int_{-\infty}^{t-\delta}\mathbb{E}\int_{{\cal O}}\Big\{\int_{r}^{r+\delta} e^{(\frac{1}{2}\mu_{m+1}-\Lambda)(t-\hat{s})}d\hat{s}\int_{r}^{r+\delta} e^{(\frac{1}{2}\mu_{m+1}-\Lambda)(t-\hat{s})}\sum_{j=1}^{\infty}|F^j(\hat{s},Y^N(\hat{s},\cdot,x))|^2d\hat{s}\\
& &+\int_{r}^{r+\delta} e^{(\frac{1}{2}\mu_{m+1}+\Lambda)(t-\hat{s})}d\hat{s}\int_{r}^{r+\delta} e^{(\frac{1}{2}\mu_{m+1}+\Lambda)(t-\hat{s})}\sum_{j=1}^{\infty}|F^j(\hat{s},Y^N(\hat{s},\cdot,x))|^2d\hat{s}\Big\}dxdr\\
&\leq & 24N^2\sigma_j^2 \|F\|^2_{\infty}vol({\cal O}) \int_{-\infty}^{t-\delta}e^{(\frac{1}{2}\mu_{m+1}-\Lambda)(t-\delta-r)}\int_{-\infty}^{r+\delta} e^{(\frac{1}{2}\mu_{m+1}-\Lambda)(r+\delta-\hat{s})}d\hat{s}\\
&& +24N^2\sigma_j^2 \|F\|^2_{\infty}vol({\cal O}) \int_{-\infty}^{t-\delta}e^{(\frac{1}{2}\mu_{m+1}+\Lambda)(t-\delta-r)}\int_{-\infty}^{r+\delta} e^{(\frac{1}{2}\mu_{m+1}+\Lambda)(r+\delta-\hat{s})}d\hat{s}\\
&\leq & 96N^2\sigma_j^2 \|F\|^2_{\infty}vol({\cal O}) \Big(\frac{1}{|\mu_{m+1}+2\Lambda|^2}+\frac{1}{|\mu_{m+1}-2\Lambda|^2}\Big)<\infty.
\end{eqnarray*}
Thus $\hat{K}_1<\infty$.

To consider $\hat{K}_2$ in (\ref{QSPDE}), note that when $r \leq t\leq r+\delta$, from (\ref{441}) and (\ref{442'}) we have
\begin{eqnarray*}
\hat{K}_2&=&\dfrac{e^{-2\Lambda |t|}}{|\delta|} \sum_{j=1}^{\infty} \int_{t-\delta}^{t}\mathbb{E}\int_{{\cal O}}|{\cal D}_{r+\delta}^j{\cal M}^N(Y^N)(t,\cdot,x)-{\cal D}_{r}^j{\cal M}^N(Y^N)(t,\cdot,x)|^2 dx dr\\
&\leq &\dfrac{4e^{-2\Lambda |t|}}{|\delta|}   \sum_{j=1}^{\infty} \int_{t-\delta}^{t}
 \chi_{\{j\geq m+1\}}(j)\mathbb{E}\int_{{\cal O}}\Big|\int_{-\infty}^{r} {\cal D}_{r}^j\Phi_{t-\hat{s},\hat{s}}^NP^jF(\hat{s},Y^N(\hat{s},\cdot,x))d\hat{s}\Big|^2dx\\
 &&+\dfrac{4e^{-2\Lambda |t|}}{|\delta|}   \sum_{j=1}^{\infty} \int_{t-\delta}^{t}\chi_{\{j\leq m\}}(j)\mathbb{E}\int_{{\cal O}}\Big|\int^{+\infty}_{r+\delta} {\cal D}_{r+\delta}^j\Phi_{t-\hat{s},\hat{s}}^NP^jF(\hat{s},Y^N(\hat{s},\cdot))d\hat{s}\Big|^2dxdr\\
&&+\dfrac{4e^{-2\Lambda |t|}}{|\delta|}   \sum_{j=1}^{\infty} \int_{t-\delta}^{t}\mathbb{E}\int_{{\cal O}}\Big|\int_{-\infty}^t \sum_{k=m+1}^{\infty}\Phi_{t-\hat{s},\hat{s}}^NP^k\nabla F(\hat{s},Y^N(\hat{s},\cdot,x))({\cal D}_{r+\delta}^j-{\cal D}_{r}^j)(Y^N)(\hat{s},\cdot,x)d\hat{s}\Big|^2dxdr\\
&&+\dfrac{4e^{-2\Lambda |t|}}{|\delta|}   \sum_{j=1}^{\infty} \int_{t-\delta}^{t}\mathbb{E}\int_{{\cal O}}\Big|\int^{+\infty}_t \sum_{k=1}^{m}\Phi_{t-\hat{s},\hat{s}}^NP^k\nabla F(\hat{s},Y^N(\hat{s},\cdot,x))({\cal D}_{r+\delta}^j-{\cal D}_{r}^j)(Y^N)(\hat{s},\cdot,x)d\hat{s}\Big|^2dxdr\\
&=&\sum_{i=4}^{7}Q_i,
\end{eqnarray*}
where
\begin{eqnarray*}
Q_{4}&\leq & \dfrac{4}{|\delta|} e^{-2\Lambda |t|} \sum_{j=m+1}^{\infty}\int_{t-\delta}^{t} \mathbb{E}\int_{{\cal O}}\Big(\int_{-\infty}^{r} \|{\cal D}_{r}^j\Phi_{t-\hat{s},\hat{s}}^NP^j\||F^j(\hat{s},Y^N(\hat{s},\cdot,x))|d\hat{s}\Big)^2dxdr\\
&\leq & \dfrac{16}{|\delta|}e^{-2\Lambda |t|}\sum_{j=m+1}^{\infty}\int_{t-\delta}^{t}\Big(\int_{-\infty}^{t} e^{\frac{1}{2}\mu_{j}(t-\hat{s})}e^{\Lambda|\hat{s}|}|F^j(\hat{s},Y^N(\hat{s},\cdot,x))|d\hat{s}\Big)^2dxdr\\
&\leq & 128\sigma_j^2 \|F\|_{\infty}^2vol({\cal O})\Big(\dfrac{1}{|\mu_{m+1}-2\Lambda|^2}+\dfrac{1}{|\mu_{m+1}+2\Lambda|^2}\Big)<\infty.
\end{eqnarray*}
and similarly,
\begin{eqnarray*}
Q_{5}&\leq & 128\sigma_j^2 \|F\|_{\infty}^2vol({\cal O})\Big(\dfrac{1}{|\mu_{m}-2\Lambda|^2}+\dfrac{1}{|\mu_{m}+2\Lambda|^2}\Big)<\infty.
\end{eqnarray*}
Furthermore, we have by the similar calculations in $Q_1$ and $Q_2$,
\begin{eqnarray*}
Q_6&\leq & e^{2\Lambda \tau}\frac{32N^2\|\nabla F\|^2_{\infty}}{|\mu_{m+1}|}\Big(\frac{1}{|\mu_{m+1}+4\Lambda|}+\frac{1}{|\mu_{m+1}-4\Lambda|}\Big) \\
&&\cdot\sup_{\hat{s}\in [0,\tau)}\dfrac{e^{-2\Lambda|\hat{s}|}}{|\delta|}\sum_{j=1}^{\infty}\int_{\mathbb{R}}\mathbb{E}\int_{{\cal O}} |({\cal D}_{r+\delta}^j-{\cal D}_{r}^j)(Y^N)(\hat{s},\cdot,x)|^2dxdr<\infty,
\end{eqnarray*}
and 
\begin{eqnarray*}
Q_{7}&\leq & e^{2\Lambda \tau}\frac{32N^2\|\nabla F\|^2_{\infty}}{|\mu_{m}|}\Big(\frac{1}{|\mu_{m}+4\Lambda|}+\frac{1}{|\mu_{m}-4\Lambda|}\Big) \\
&&\cdot\sup_{\hat{s}\in [0,\tau)}\dfrac{e^{-2\Lambda|\hat{s}|}}{|\delta|}\sum_{j=1}^{\infty}\int_{\mathbb{R}}\mathbb{E}\int_{{\cal O}} |({\cal D}_{r+\delta}^j-{\cal D}_{r}^j)(Y^N)(\hat{s},\cdot,x)|^2dxdr<\infty.
\end{eqnarray*}
Thus we have that $\hat{K}_2<\infty$.

Finally to consider $\hat{K}_3$, note that when $t \leq r$, (\ref{441}) and (\ref{442}) gives us 
\begin{eqnarray*}
\hat{K}_3&=&\dfrac{e^{-2\Lambda |t|}}{|\delta|}\sum_{j=1}^{\infty}  \int_{t}^{+\infty}\mathbb{E}\int_{{\cal O}}|{\cal D}_{r+\delta}^j{\cal M}^N(Y^N)(t,\cdot,x)-{\cal D}_{r}^j{\cal M}^N(Y^N)(t,\cdot,x)|^2dxdr\\
&\leq &\dfrac{3e^{-2\Lambda |t|}}{|\delta|} \sum_{j=1}^{\infty}\int_{t}^{+\infty} \mathbb{E}\int_{{\cal O}}\bigg\{\Big|\int_{-\infty}^t \sum_{k=m+1}^{\infty}\Phi_{t-\hat{s},\hat{s}}^NP^k\nabla F(\hat{s},Y^N(\hat{s},\cdot,x))({\cal D}_{r+\delta}^j-{\cal D}_{r}^j)(Y^N)(\hat{s},\cdot,x)d\hat{s}\Big|^2\\
&&+\chi_{\{j\leq m\}}(j)\Big|\int^{+\infty}_{r+\delta} {\cal D}_{r+\delta}^j\Phi_{t-\hat{s},\hat{s}}^NP^jF(\hat{s},Y^N(\hat{s},\cdot,x))d\hat{s}-\int^{+\infty}_{r} {\cal D}_{r}^j\Phi_{t-\hat{s},\hat{s}}^NP^jF(\hat{s},Y^N(\hat{s},\cdot,x))d\hat{s}\Big|^2\\
&&+\Big|\int^{+\infty}_t \sum_{k=1}^{m}\Phi_{t-\hat{s},\hat{s}}^NP^k\nabla F(\hat{s},Y^N(\hat{s},\cdot,x))({\cal D}_{r+\delta}^j-{\cal D}_{r}^j)(Y^N)(\hat{s},\cdot,x)d\hat{s}\Big|^2\bigg\}dxdr\\
&=&\sum_{i=8}^{10}Q_i,
\end{eqnarray*}
Now it is easy to write down the following estimations,
\begin{eqnarray*}
Q_8&\leq & e^{2\Lambda \tau}\frac{24N^2\|\nabla F\|^2_{\infty}}{|\mu_{m+1}|}\Big(\frac{1}{|\mu_{m+1}+4\Lambda|}+\frac{1}{|\mu_{m+1}-4\Lambda|}\Big)\\
&&\cdot\sup_{\hat{s}\in [0,\tau)}\dfrac{e^{-2\Lambda|\hat{s}|}}{|\delta|}\sum_{j=1}^{\infty}\int_{\mathbb{R}}\mathbb{E}\int_{{\cal O}} |({\cal D}_{r+\delta}^j-{\cal D}_{r}^j)(Y^N)(\hat{s},\cdot,x)|^2dxdr<\infty,
\end{eqnarray*}
and
\begin{eqnarray*}
Q_{10}&\leq & e^{2\Lambda \tau}\frac{24N^2\|\nabla F\|^2_{\infty}}{|\mu_{m}|}\Big(\frac{1}{|\mu_{m}+4\Lambda|}+\frac{1}{|\mu_{m}-4\Lambda|}\Big) \\
&&\cdot\sup_{\hat{s}\in [0,\tau)}\dfrac{e^{-2\Lambda|\hat{s}|}}{|\delta|}\sum_{j=1}^{\infty}\int_{\mathbb{R}}\mathbb{E}\int_{{\cal O}} |({\cal D}_{r+\delta}^j-{\cal D}_{r}^j)(Y^N)(\hat{s},\cdot,x)|^2dxdr<\infty.
\end{eqnarray*}
Similarly to $Q_2$,
\begin{eqnarray*}
Q_{9} &\leq & 96N^2\sigma_j^2 \|F\|^2_{\infty}vol({\cal O}) \Big(\frac{1}{|\mu_{m}+2\Lambda|^2}+\frac{1}{|\mu_m-2\Lambda|^2}\Big)<\infty.
\end{eqnarray*}

In summary, we have shown that
$$\sup\limits_{\substack{t\in (0,\tau]}}\dfrac{e^{-2\Lambda |t|}}{|\delta|} \sum_{j=1}^{\infty}\int_{\mathbb{R}}\mathbb{E}\int_{{\cal O}}|{\cal D}_{r+\delta}^j\mathcal{M}^N(Y^N)(t,\cdot,x)-{\cal D}_{r}^j\mathcal{M}^N(Y^N)(t,\cdot,x)|^2dxdr< \infty.$$
This leads to the conclusion that ${\cal M}^N$ maps from $C^{\Lambda,N}_{\tau,\rho}(\mathbb{R},L^2({\cal O},{\cal D}^{1,2}))$ to itself.

{\it \textbf{Step 2}}: Now we are ready to conclude that that for each $N\in \mathbb{N}$, $\mathcal{M}^N(C^{\Lambda,N}_{\tau,\rho}(\mathbb{R}, L^2({\cal O},{\cal D}^{1,2}))|_{(0,\tau]}$ is relatively compact in $C((0,\tau], L^2(\Omega\times{\cal O}))$ . This can be easily achieved by applying \cite[Theorem 2.3]{feng-zhao} and results from {\it \textbf{Step 1}}.  
\end{proof}
\begin{proof}[of Theorem \ref{THMC4M}]
From Lemma \ref{LEMMA4C4}, we know that ${\cal M}^N( S)\big|_{(0,\tau]}$ is relatively compact in $C_\tau^{\Lambda}((0,\tau], L^2(\Omega\times {\cal O}))$, where $$S:=C^{\Lambda,N}_{\tau,\rho}((0,\tau],L^2(\Omega, {\cal D}^{1,2}))\cap L_{\Lambda}^{\infty}(\mathbb{R}, L^2(\Omega, H^1({\cal O}))),$$
i.e., for any sequence ${\cal M}^N(Y^N_n)\in
C^{\Lambda,N}_{\tau,\rho}(\mathbb{R},L^2({\cal O},{\cal D}^{1,2}))$, there exists a
subsequence, still denoted by ${\cal M}^N(Y^N_n)$, and $V^N\in
C([0,\tau), L^2(\Omega\times{\cal O}))$ such that
\begin{eqnarray}\label{zhao1}
\sup\limits_{t\in (0,\tau]}\mathbb{E}\int_{{\cal O}}|{\cal M}^N(Y^N_n)(t,\cdot, x)-V^N(t,\cdot, x)|^2dx\to 0
\end{eqnarray}
as $n\to \infty$. 

 Set for any $t\in [m\tau, m\tau+\tau)$, $$V^N(t,\omega, x)=V^N(t-m\tau,\theta_{m\tau}\omega, x).$$
Note that by definition
\begin{eqnarray*}
{\cal M}^N(Y^N_n)(t,\theta_{m\tau}\omega, x)={\cal M}^N(Y^N_n)(t+m\tau,\omega, x).
\end{eqnarray*}
With (\ref{zhao1}) and the probability preserving of $\theta$, we get 
\begin{eqnarray*}
&&\sup\limits_{t\in [m\tau, m\tau+\tau)}e^{-2\Lambda|t|}\mathbb{E}\int_{{\cal O}}|{\cal M}^N(Y^N_n)(t,\cdot, x)-V^N(t,\cdot, x)|^2dx\\
&=&\sup\limits_{t\in [0,\tau)}e^{-2\Lambda|t-m\tau|}\mathbb{E}\int_{{\cal O}}|{\cal M}^N(Y^N_n)(t+m\tau ,\cdot, x)-V^N(t+m\tau,\cdot, x)|^2dx\\
&\leq &\sup\limits_{t\in [0,\tau)}\mathbb{E}\int_{{\cal O}}|{\cal M}^N(Y^N_n)(t,\theta
_{m\tau}\cdot, x)-V^N(t,\theta_{m\tau}\cdot, x)|^2dx\\
&=&\sup\limits_{t\in [0,\tau)}\mathbb{E}\int_{{\cal O}}|{\cal M}^N(Y^N_n)(t,\cdot, x)-V^N(t,\cdot, x)|^2dx\\
&\to& 0,
\end{eqnarray*}
 Thus
\begin{eqnarray*}
\sup\limits_{t\in \mathbb{R}}e^{-2\Lambda|t|}\mathbb{E}\int_{{\cal O}}|{\cal M}^N(Y^N_n)(t,
\cdot, x)-V^N(t,\cdot, x)|^2dx\to 0,
\end{eqnarray*}
as $n\to \infty$.
Therefore ${\cal M}^N(S)$ is relatively compact in $C_\tau^{\Lambda}(\mathbb{R}, L^2(\Omega\times{\cal O}))$.

According to the generalized
Schauder's fixed point argument Theorem \ref{Schauder}, ${\cal M}^N$ has a fixed point in
$C_{\tau}^{\Lambda}(\mathbb{R}, L^2(\Omega\times {\cal O}))$. That is to say
there exists a solution $Y^N\in
C_{\tau}^{\Lambda}(\mathbb{R}, L^2(\Omega\times{\cal O}))$ of equation
(\ref{VN}) such that for any $t\in \mathbb{R}$,
$Y^N(t+\tau,\omega, x)=Y^N(t,\theta_{\tau}\omega, x)$.   
\end{proof}

\begin{thm}\label{Main2}
 Let $F:\mathbb{R} \times \mathbb{R}^d\to \mathbb{R}$ be a continuous map, globally
bounded and the Jacobian $\nabla F(t,\cdot)$ be globally bounded.
Assume $F(t,u)=F(t+\tau,u)$ for some fixed $\tau>0$, and  Condition (L) and Condition (B) hold. There exists at least one ${\cal B}(\mathbb{R})\otimes\mathcal{F}$-measurable map $\hat{Y}: \mathbb{R}\times\Omega\rightarrow \mathbb{R}^d$
satisfying Eqn. (\ref{V}) and $\hat{Y}(t+\tau, \omega)=\hat{Y}(t, \theta_{\tau}\omega)$ for
any $t\in \mathbb{R}$ and $\omega\in \Omega$.
\end{thm}
\begin{proof}
 Now define a subset of $\Omega$ as
$$\Omega_N:=\left\{\omega,\ C_{\Lambda}(\omega)<N\right\}.$$
As the random variable $C_{\Lambda}(\omega)$ is tempered from above, it is easy to get
\begin{equation*}
\mathbb{P}(\Omega_N)\to 1,
\end{equation*}
as $N \to \infty$. Note also $\Omega_N$ is an increasing sequence of sets, thus $\cup_{n}\Omega_N=\hat{\Omega}$ and $\hat{\Omega}$ has full measure, and is invariant with respect to $\theta$.
Then define
\begin{equation*}
\Omega^{\ast}_N=\bigcup_{n=-\infty}^{\infty}\theta_{n\tau}^{-1}\Omega_N,
\end{equation*}
and $\Omega^{\ast}_N$ is invariant with respect to $\theta_{n\tau}$ for each n. Besides we have $\Omega^{\ast}_N\subset \Omega^{\ast}_{N+1}$, which leads to
\begin{eqnarray*}
\bigcup_{N}\Omega^{\ast}_N=\bigcup_{N}\bigcup_{n=-\infty}^{\infty}\theta_{n\tau}^{-1}\Omega_N=\bigcup_{n=-\infty}^{\infty}\theta_{n\tau}^{-1}\left(\bigcup_{N}\Omega_N\right)=\bigcup_{n=-\infty}^{\infty}\theta_{n\tau}^{-1}\hat{\Omega}=\bigcup_{n=-\infty}^{\infty}\hat{\Omega}=\hat{\Omega},
\end{eqnarray*}
with $\mathbb{P}(\hat{\Omega})=1$.

Now we can define $Y:\hat{\Omega}\times\mathbb{R}\to L_0^2({\cal O})$, as a combinations of $Y_N$ such that
\begin{eqnarray}\label{NJ521}
Y:=Y_1\chi_{\Omega^{\ast}_1}+Y_2\chi_{\Omega^{\ast}_2\setminus\Omega^{\ast}_1}+\cdots+Y_N\chi_{\Omega^{\ast}_N\setminus\Omega^{\ast}_{N-1}}+\cdots,
\end{eqnarray}
and it is easy to see that $Y$ is $\mathcal{B}(\mathbb{R})\otimes\mathcal{F}$ measurable and thus Y also satisfies the following property
\begin{eqnarray*}
Y(t+\tau,\omega)&=&Y_1(t+\tau,\omega)\chi_{\Omega^{\ast}_1}(\omega)+Y_2(t+\tau,\omega)\chi_{\Omega^{\ast}_2\setminus\Omega^{\ast}_1}(\omega)+\cdots+Y_N(t+\tau,\omega)\chi_{\Omega^{\ast}_N\setminus\Omega^{\ast}_{N-1}}(\omega)+\cdots\\
&=&Y_1(t,\theta_{\tau}\omega)\chi_{\Omega^{\ast}_1}(\omega)+Y_2(t,\theta_{\tau}\omega)\chi_{\Omega^{\ast}_2\setminus\Omega^{\ast}_1}(\omega)+\cdots+Y_N(t,\theta_{\tau}\omega)\chi_{\Omega^{\ast}_N\setminus\Omega^{\ast}_{N-1}}(\omega)+\cdots\\
&=&Y_1(t,\theta_{\tau}\omega)\chi_{\Omega^{\ast}_1}(\theta_{\tau}\omega)+Y_2(t,\theta_{\tau}\omega)\chi_{\Omega^{\ast}_2\setminus\Omega^{\ast}_1}(\theta_{\tau}\omega)+\cdots+Y_N(t,\theta_{\tau}\omega)\chi_{\Omega^{\ast}_N\setminus\Omega^{\ast}_{N-1}}(\theta_{\tau}\omega)+\cdots\\
&=&Y(t,\theta_{\tau}\omega).
\end{eqnarray*}
Moreover $Y$ is a fixed point of $\mathcal{M}$.

We can easily extend $Y$ to the whole $\Omega$ as $\mathbb{P}(\hat{\Omega})=1$, which is distinguishable with $Y$ defined in (\ref{NJ521}).  
\end{proof}

\section{Stochastic Allen-Cahn equation with periodic perturbation}\label{sec:5}
In this section, we relax the boundedness condition on the nonlinear drift as an extension of the results that we have developed in this paper. 
The equation considered in this section is the stochastic Allen-Cahn equation with periodic forcing where only a weakly dissipative condition is
satisfied for the nonlinear term. We introduce a sequence of SPDEs with coefficients as the cut-off of the coefficients in the original Allen-Cahn equation and apply Theorem \ref{Main} and Theorem \ref{THMC4M} to the truncated equation. Then we conclude the existence of random periodic solution to the original SPDE through a localization argument. 

Consider the following SPDE
\begin{eqnarray}\label{eqn:A-C eqn}
\Bigg\{\begin{array}{l}du(t,x)=\Delta u(t,x)\,dt+F(t,u(t,x))dt+Bu(t,x)dW(t),  \ \ \ \ \ \ t\geq s \\
u(s)=\psi\in L^2({\cal O}),\\
u(t)|_{\partial {\cal O}}=0.
                                     \end{array}
\end{eqnarray}
where $F(t,u)$ satisfies Condition (P) with period $\tau>0$, continuity in $t$ and $C^1$ in $u$, and a weakly dissipative condition 
\begin{eqnarray}\label{eqn4.1}
uF(t,u)\leq -M u^2+L,
\end{eqnarray}
for some constant positive constant $M$ and $L$, 
${\cal O}$ is a bounded open subset of $\mathbb{R}^d$ with smooth boundary, $B$ satisfies Condition (B), $W(t)$ is a $L^2({\cal O})$-valued Brownian motion on a probability space $(\Omega,{\cal F},\mathbb{P})$ as in the previous sections. We assume that $M>{\sigma^2\over 2}$, where $\sigma^2:=\max_i \sigma_i^2$.

\begin{prop}\label{prop5.1}
 Under the above conditions, SPDE (\ref{eqn:A-C eqn}) has a random periodic solution of period $\tau$. 
\end{prop}
\begin{proof}
For any $N\in {\mathbb N}$, define $F^N$ to be a $C^1$ function such that when $u^2<2^N$, $F^N(t,u)= F(t,u)$,  when $u<-\sqrt{2^{N}+1}$, $F^N(t, u)=F(t, -\sqrt{2^N+1})$ and when $u>\sqrt{2^{N}+1}$, $F^N(t, u)=F(t, \sqrt{2^N+1})$. This can be seen from the partition of unity. It is easy to see that $F^N$ is bounded with bounded derivative with respect to $u$. 
Then by Theorem \ref{Main} and Theorem \ref{THMC4M}, the cut-off SPDE 
\begin{eqnarray*}
\Bigg\{\begin{array}{l}du(t,x)=\Delta u(t,x)\,dt+F^N(t, u(t,x))\, dt+Bu(t,x)dW(t),  \ \ \ \ \ \ t\geq s \\
u(s)=\psi\in L^2({\cal O}),\\
u(t)|_{\partial {\cal O}}=0,
                                     \end{array}
\end{eqnarray*}
has a random periodic solution $Y^N$ which satisfies the integral equation
\begin{eqnarray}\label{eqn4.2}
Y^N(t,\omega)=\int_{-\infty}^t \Phi (t-s,\theta_s\omega)F^N(s, Y^N(s,\omega))ds.
\end{eqnarray}
Now we prove that $Y^N$ is uniformly bounded in $L^2({\cal O}\times\Omega)$. 
Denote $Y^{N,i}_r:=<Y^N_r, \phi_i>_{\mathbb H}$. It is easy to see that 
$$\int_{\cal O} |Y_r^N|^2dx=\int_{\cal O}|\sum_iY^{N,i}_r\phi_i(x)|^2 dx= \sum_{i}|Y^{N,i}_r|^2.$$
For any real value $K$, applying It$\rm\hat{o}$'s formula to $e^{Kr}|Y^N_r|^2$
 and using the estimate (\ref{eqn4.1}), we have
\begin{eqnarray*}
&&\int_{\cal O}\mathbb{E} [e^{Kt}|Y^N_t|^2]dx\\
&=&\mathbb{E} \int_{\cal O}\int_{-\infty}^t \left [Ke^{Kr}|Y^N_r|^2+2 Y^N_r\Delta Y^N_r e^{Kr}+2e^{Kr} Y^N_rF^N(r, Y^N(r,\omega))+e^{Kr}|\sum_{i=1}^\infty \sigma_i Y^{N,i}_r\phi_i(x)|^2\right]drdx\\
&\leq & \mathbb{E} \int_{\cal O}\int_{-\infty}^t  e^{Kr}[K|Y^N_r|^2+2 Y^N_r\Delta Y^N_r+2(-M|Y^N_r|^2+L)]drdx+\int_{-\infty}^t e^{Kr}\sum_{i=1}^\infty \sigma_i^2 |Y^{N,i}_r|^2dr\\
&\leq &\mathbb{E} \int_{-\infty}^t\int_{\cal O} e^{Kr}\Big[(K-2M+\sigma^2)|Y^N_r|^2+2L+2 Y^N_r\Delta Y_r^N\Big]drdx\\
&=& \mathbb{E}\int_{-\infty}^t \Big[e^{Kr}\cdot2L+2 e^{Kr}Y^N_r\Delta Y_r^N\Big]dr,
\end{eqnarray*}
by choosing $K=2M-\sigma^2>0$. Therefore, 
\begin{eqnarray*}
\int_{{\cal O}} \mathbb{E} |Y^N_t|^2dx&\leq&  \int_{\cal O}\int_{-\infty}^t e^{-K(t-r)}2Ldrdx+ \mathbb{E}\int_{-\infty}^t 2 e^{-K(t-r)}\int_{{\cal O}}Y^N_r\Delta Y_r^Ndxdr\\
&=&{2L\over K} m ({\cal O})+ \mathbb{E}\int_{-\infty}^t 2 e^{-K(t-r)}(-\int_{{\cal O}}|\nabla Y^N_r|^2dx)dr\\
&\leq & {2L\over K} m ({\cal O})\\
&:=&\tilde M,
\end{eqnarray*}
by using the integration by parts formula. To make the Lebesgue measure in ${\cal O}$ a probability measure, let $m({\cal O})=1$ without losing any generality. Otherwise, we can always achieve this via normalization. Note that the constant $\tilde M$ is independent of $N$.
By Chebyshev's inequality, 
\begin{eqnarray*}
(\mathbb P\times m)(|Y^N(t,x)|^2>2^n) \leq {1\over {2^n}}\mathbb{E}\int_{{\cal O}} |Y^N_t|^2dx\leq {{\tilde M}\over {2^n}}.
\end{eqnarray*}
As $\sum_{n=1}^\infty {1\over {2^n}}<\infty$, we can use Borel-Cantelli lemma to get
\begin{eqnarray*}
(\mathbb P\times m)\Big(\limsup_{n\to \infty}\big\{|Y^N(t,x)|^2>2^n\big\}\Big)=0,
\end{eqnarray*}
and therefore
\begin{eqnarray*}
(\mathbb P\times m)\Big(\liminf_{n\to \infty}\big\{|Y^N(t,x)|^2\leq 2^n\big\}\Big)=1,
\end{eqnarray*}
i.e. 
\begin{eqnarray*}
(\mathbb P\times m)\Big(\bigcup_{n=1}^\infty \bigcap_{k\geq n}\big\{|Y^N(t,x)|^2\leq 2^n\big\}\Big)=1.
\end{eqnarray*}
This means for a.s. $(\omega, x)$, there exists an integer $\hat N(\omega, x)$ such that $|Y^N(t,x)|^2\leq 2^{\hat N}:=c(\omega, x)$.

Now define a subset of $\Omega\times {\cal O}$ as
$$\hat \Omega_N:=\left\{(\omega,x)\in \Omega\times {\cal O}:\ c(\omega,x)<N\right\},$$
It is easy to see that 
$(\mathbb{P}\times m)(\hat \Omega_N)\to 1$, as $N\to \infty$. Note also that $\hat \Omega_N$ is an increasing sequence of sets, thus $\cup_{N}\hat \Omega_N=:\hat{\Omega}$ and $\hat{\Omega}$ has full measure, and is invariant with respect to $\theta$.
Define
\begin{equation*}
\hat \Omega^{\ast}_N=\bigcup_{n=-\infty}^{\infty}\theta_{n\tau}^{-1}\hat \Omega_N,
\end{equation*}
then $\hat \Omega^{\ast}_N$ is invariant with respect to $\theta_{n\tau}$ for each n. Besides we have $\hat \Omega^{\ast}_N\subset \hat \Omega^{\ast}_{N+1}$, which leads to
\begin{eqnarray*}
\bigcup_{N}\hat \Omega^{\ast}_N=\bigcup_{N}\bigcup_{n=-\infty}^{\infty}\theta_{n\tau}^{-1}\hat{\Omega}_N=\bigcup_{n=-\infty}^{\infty}\hat{\Omega}=\hat{\Omega},
\end{eqnarray*}
with $(\mathbb{P}\times m)(\hat{\Omega})=1$. For any $(\omega, x)\in \hat \Omega$, there exists $N\in\mathbb N$ such that $(\omega,x)\in \hat \Omega _N$. 

As $(\mathbb{P}\times m)(\hat{\Omega^c})=0$, this will lead to
$$\mathbb{P}((\hat \Omega^c)_x)=0\ {\rm and}\ m((\hat\Omega ^c)_\omega)=0,$$
where 
$$(\hat \Omega^c)_x:= \{\omega: (\omega, x)\in \hat \Omega ^c\},\ (\hat \Omega^c)_\omega:= \{x: (\omega, x)\in \hat \Omega ^c\}.$$
Otherwise, if 
$$\mathbb{P}((\hat \Omega^c)_x)>0\ {\rm or}\ m((\hat\Omega ^c)_\omega)>0,$$
then 
$$(\mathbb{P}\times m)(\hat{\Omega^c})=\int_D \mathbb P((\hat \Omega^c)_x)m(dx)>0,$$
or
$$(\mathbb{P}\times m)(\hat{\Omega^c})=\int_\Omega m((\hat \Omega^c)_\omega)\mathbb P(d\omega)>0,$$
which is a contradiction. Therefore, 
$$
\mathbb{P}((\hat \Omega)_x)=1\ {\rm and}\ m((\hat\Omega )_\omega)=1.
$$ 
Now we define 
$$
(\hat \Omega_N)_x=\{\omega: (\omega,x)\in \hat \Omega_N\}.
$$
Then first it is obvious that $(\hat \Omega_N)_x\subset (\hat \Omega_{N+1})_x$. Moreover, if 
$\omega \in \hat \Omega _x$, so by definition, $(\omega,x)\in \hat \Omega$. 
So there exists an integer $N$ such that $(\omega,x)\in \hat \Omega_N^*$. Thus $\omega\in (\hat \Omega_N^*)_x$.
This means 
$$
(\hat \Omega)_x
=\bigcup_{N}(\hat \Omega_N^*)_x.
$$
Now we can define $Y:\hat{\Omega}\times\mathbb{R}\to L^2({\cal O})$, as a combinations of $Y^N$ such that
\begin{eqnarray}\label{NJ521}
Y(t,\omega)(x):&=&Y^1(t,\omega)(x)\chi_{(\hat \Omega^{\ast}_1)_x}(\omega)+Y^2(t,\omega)(x)\chi_{(\hat \Omega^{\ast}_2)_x\setminus(\hat \Omega^{\ast}_1)_x}(\omega)+\cdots\nonumber
\\
&&
+Y^N(t,\omega)(x)\chi_{(\hat \Omega^{\ast}_N)_x\setminus(\hat\Omega^{\ast}_{N-1})_x}(\omega)
+\cdots,
\end{eqnarray}
and it is easy to see that $Y$ is $\mathcal{B}((-\infty, T))\otimes\mathcal{F}$ measurable and thus Y also satisfies the following property
\begin{eqnarray*}
&&
Y(t+\tau,\omega)\\
&=&Y^1(t+\tau,\omega)\chi_{(\hat\Omega^{\ast}_1)_{\cdot}}(\omega)+Y^2(t+\tau,\omega)\chi_{(\hat\Omega^{\ast}_2)_{\cdot}\setminus(\hat \Omega^{\ast}_1)_{\cdot}}(\omega)+\cdots+Y^N(t+\tau,\omega)\chi_{(\hat \Omega^{\ast}_N)_{\cdot}\setminus(\hat \Omega^{\ast}_{N-1})_{\cdot}}(\omega)+\cdots\\
&=&Y^1(t,\theta_{\tau}\omega)\chi_{(\hat \Omega^{\ast}_1)_{\cdot}}(\omega)+Y^2(t,\theta_{\tau}\omega)\chi_{(\hat 
\Omega^{\ast}_2)_{\cdot}\setminus(\hat \Omega^{\ast}_1)_{\cdot}}(\omega)+\cdots+Y^N(t,\theta_{\tau}\omega)\chi_{(\hat 
\Omega^{\ast}_N)_{\cdot}\setminus(\hat \Omega^{\ast}_{N-1})_{\cdot}}(\omega)+\cdots\\
&=&Y^1(t,\theta_{\tau}\omega)\chi_{(\hat \Omega^{\ast}_1)_{\cdot}}(\theta_{\tau}\omega)+Y^2(t,\theta_{\tau}\omega)\
\chi_{(\hat \Omega^{\ast}_2)_{\cdot}\setminus(\hat \Omega^{\ast}_1)_{\cdot}}(\theta_{\tau}\omega)+\cdots+Y^N(t,\theta_{\tau}
\omega)\chi_{(\hat \Omega^{\ast}_N)_{\cdot}\setminus(\hat \Omega^{\ast}_{N-1})_{\cdot}}(\theta_{\tau}\omega)+\cdots\\
&=&Y(t,\theta_{\tau}\omega).
\end{eqnarray*}
 We can extend $Y$ to the whole $\Omega$ , i.e. for a.s. $\omega\in \Omega$, 
$$Y(t,\omega, x)=Y^N(t,\omega, x), \ (\omega, x)\in \Omega_N^*\setminus\Omega_{N-1}^*$$
holds for a.s. $x\in {\cal O}$. So for a.s. $\omega\in \Omega$, $Y$ satisfies (\ref {eqn4.2}) for a.s. $x\in {\cal O}$. This means $Y$ is a weak solution for a.s. $\omega\in \Omega$.
\end{proof}
\begin{exmp}
SPDE (\ref{eqn:A-C eqn}) will become the stochastic Allen-Cahn equation with periodic perturbation when we takie $F(t,u):= u-u^3+\sin t$ in (\ref{eqn:A-C eqn}).
For this $F$, the weakly dissipative condition is satisfied as for any $M>0$,
\begin{eqnarray*}
uF(t,u)&=&u^2-u^4+u\sin t\\
&=&(u^2-u^4)I_{\{u^2-1>M\}}+(u^2-u^4)I_{\{u^2-1\leq M\}}+\epsilon u^2+{1\over \epsilon}\sin^2 t\\
&\leq& -Mu^2+(1+M)+\epsilon u^2+{1\over \epsilon}\\
&=& -\tilde M u^2+L,
\end{eqnarray*}
where $\tilde M:=M-\epsilon$ and $L:=1+M+{1\over \epsilon}$. We can always choose $M$ to be big enough such that $\tilde M>{\sigma^2 \over 2}$.
So by Proposition \ref{prop5.1}, we know that the the stochastic Allen-Cahn equation has a random periodic solution of period $2\pi$.
\end{exmp}

\section{Appendix}
\begin{proof}[Proof of Corollary \ref{LEMMA2C4}]
First note that
\begin{eqnarray*}
\|T_{t}P^k-P^k\|^2&=&\sup_{\|v\|_H=1}\int_{{\cal O}}|(T_tP^k-P^k)v(x)|^2dx=\sup_{\|v\|_H=1}\int_{{\cal O}}|\langle(e^{\mu_kt}-1)\phi_k(\cdot),v(\cdot)\rangle\phi_k(x)|^2dx\\
&=&\sup_{\|v\|_H=1}(1-e^{\mu_kt})^2|v^k(\cdot)|^2<\sup_{\|v\|_H=1}(1-e^{\mu_kt})|v^k(\cdot)|^2<|\mu_k||t|,
\end{eqnarray*}
which immediately implies (\ref{LEMMA2C41}).

To show (\ref{LEMMA2C42}), let us check with the following SPDEs of $\Phi_t$ derived from (\ref{phi}),
\begin{eqnarray*}
\Bigg\{\begin{array}{l} d\Phi_tP^k=\mathcal{L}\Phi_tP^kdt+B(\Phi_tP^k)\circ dW^k_t,\\
\\
\Phi_0P^k=P^k.
\end{array}
\end{eqnarray*}
It is easy to see that for $t\geq 0$ and $k\geq m+1$,
\begin{eqnarray*}
\Phi_tP^k=T_tP^k+\int_0^tT_{t-s}B(\Phi_sP^k)\circ dW^k_s.
\end{eqnarray*}
Thus we have
\begin{eqnarray*}
\mathbb{E}\|\Phi_{t}P^k-T_tP^k\|^2
&=&\mathbb{E}\Big\|\int_0^{t}T_{t-h}P^kB\Phi_hP^k\circ dW_h\Big\|^2\\
&= &\sigma_k^2 \mathbb{E}\int_{{\cal O}}\Big|\int_0^{t}T_{t-h}P^k\Phi_hP^k(\phi_k)(x)\circ dW^k_h\Big|^2dx
\end{eqnarray*}
Note that the relation between Stratonovich integral and It\^o integral, i.e.,
\begin{align*}
&T_tP^k\int_0^{t}T_{-h}P^k\Phi_hP^k(\phi_k)(x)\circ dW^k_h\\
= & T_tP^k\int_0^{t}T_{-h}P^k\Phi_hP^k(\phi_k)(x) dW^k_h+ \frac{1}{2}T_tP^k\int_0^{t}(T_{-h}P^k)^2\Phi_hP^k(\phi_k)(x) d h,
\end{align*}
which leads to the following estimates
\begin{eqnarray*}
& & \sigma_k^2 \mathbb{E}\int_{{\cal O}}\Big|\int_0^{t}T_{t-h}P^k\Phi_hP^k(\phi_k)(x)\circ dW^k_h\Big|^2dx\\
&\leq &2\sigma_k^2 \|T_tP^k\|^2\Big(\int_{{\cal O}}\int_0^{t}\|T_{-h}P^k\|^2\mathbb{E}|\Phi_hP^k(\phi_k)(x)|^2 dhdx 
+\int_{{\cal O}}\mathbb{E}\Big|\int_0^{t}\|T_{-h}P^k\|^2|\Phi_hP^k(\phi_k)(x)| dh\Big|^2dx \Big)\\
&\leq &C\max\big\{1,e^{2\sigma^2_k |t|}\big\}\sigma_k^2 (|t|+|t|^2).
\end{eqnarray*} 
Here we used the following estimate
\begin{eqnarray*}
\mathbb{E}|\Phi_hP^k(\phi_k)(x)|^2\leq \mathbb{E}e^{2\mu_kh+2\sigma_k W_h^k}|(\phi_k)(x)|^2\leq e^{2\mu_kh+2\sigma^2_k  h}|(\phi_k)(x)|^2.
\end{eqnarray*}
Thus we prove (\ref{LEMMA2C42}) for the case $t\geq 0$ and $k\geq m+1$. The case for $t<0$ and $1\leq k\leq m$ can be derived analogously.  
\end{proof}


\section*{Acknowledgements}

 CF and HZ would like to acknowledge the financial support of Royal Society Newton Advanced Fellowship NA150344.

\bibliographystyle{siamplain}
\bibliography{references}

   \end{document}